\documentclass[11pt,reqno]{amsart}
\usepackage{mathrsfs}
\usepackage{url}
\usepackage{mathtools}
\usepackage{latexsym,epsfig,amssymb,amsmath,amsthm,color,url,bm}
\usepackage[inline,shortlabels]{enumitem}
\usepackage{hyperref}
\usepackage[foot]{amsaddr}
\usepackage{amsmath,amsbsy}
\usepackage{mwe}
\RequirePackage[numbers]{natbib}
\usepackage{mathptmx}
\usepackage[text={16cm,24cm}]{geometry}

\allowdisplaybreaks 
 \numberwithin{equation}{section}
\newtheorem{theorem}{Theorem}[section]
\newtheorem{prop}[theorem]{Proposition}
\newtheorem{lemma}[theorem]{Lemma}

\newcommand\Item[1][]{%
  \ifx\relax#1\relax  \item \else \item[#1] \fi
  \abovedisplayskip=0pt\abovedisplayshortskip=0pt~\vspace*{-\baselineskip}}

\theoremstyle{definition}
\newtheorem{defn}[theorem]{Definition}

\theoremstyle{definition}
\newtheorem{remark}[theorem]{Remark}

\DeclareMathOperator{\Prob}{\mathbf{P}}

\DeclareMathOperator{\F}{\widehat{F}}
\DeclareMathOperator{\f}{\widehat{f}}
\DeclareMathOperator{\Out}{Out}

\title[On a class of PCA  with size-$3$ neighbourhood and their applications in percolation games]{On a class of probabilistic cellular automata with size-$3$ neighbourhood and their applications in percolation games}
\date{}
\author{Dhruv Bhasin, Sayar Karmakar, Moumanti Podder, Souvik Roy}
\address{Dhruv Bhasin, Indian Institute of Science Education and Research (IISER) Pune, Dr.\ Homi Bhabha Road, Pashan, Pune 411008, Maharashtra, India.}
\address{Sayar Karmakar, University of Florida, 230 Newell Drive, Gainesville, Florida 32605, USA.}
\address{Moumanti Podder, Indian Institute of Science Education and Research (IISER) Pune, Dr.\ Homi Bhabha Road, Pashan, Pune 411008, Maharashtra, India.}
\address{Souvik Roy, Indian Statistical Institute, 203 Barrackpore Trunk Road, Kolkata 700108, West Bengal, India.}
\email{bhasin.dhruv@students.iiserpune.ac.in}
\email{sayarkarmakar@ufl.edu}
\email{moumanti@iiserpune.ac.in}
\email{souvik.2004@gmail.com}

\begin{document}
\bibliographystyle{plainnat}

\begin{abstract}
Different versions of percolation games on $\mathbb{Z}^{2}$, with parameters $p$ and $q$ that indicate, respectively, the probability with which a site in $\mathbb{Z}^{2}$ is labeled a trap and the probability with which it is labeled a target, are shown to have probability $0$ of culminating in draws when $p+q > 0$. We show that, for fixed $p$ and $q$, the probability of draw in each of these games is $0$ if and only if a certain $1$-dimensional probabilistic cellular automaton (PCA) $F_{p,q}$ with a size-$3$ neighbourhood is ergodic. This allows us to conclude that $F_{p,q}$ is ergodic whenever $p+q > 0$, thereby rigorously establishing ergodicity for a considerable class of PCAs.
\end{abstract}



\maketitle

\section{Introduction}
\subsection{Overview of the paper}\label{subsec:overview} There are two aspects to this paper that share a deep connection with each other: a class of $1$-dimensional \emph{probabilistic cellular automata} (henceforth abbreviated as PCA), and a collection of \emph{percolation games} that are studied on the $2$-dimensional lattice $\mathbb{Z}^{2}$. At the very outset, we give a brief sketch of the specific PCAs and the versions of percolation games that we focus on, while their detailed descriptions are provided in \S\ref{subsec:PCA} and \S\ref{subsec:game} respectively. 

The key components constituting our PCA (that we henceforth refer to as $F_{p,q}$) are the alphabet $\{0,1\}$, the universe $\mathbb{Z}$, the neighbourhood $\mathcal{N} = \{0,1,2\}$, and the parameters $p$ and $q$ with $(p,q) \in \mathcal{S}$, where 
\begin{equation}\label{mathcal{S}_defn}
\mathcal{S} = \{(p',q') \in [0,1]^{2}: 0 < p'+q' \leqslant 1\}.
\end{equation}
A \emph{configuration} for this PCA at a given time $t$ is an element $\eta_{t}$ of $\{0,1\}^{\mathbb{Z}}$, and its evolution can be described as follows: the state $\eta_{t+1}(n)$ of a site $n \in \mathbb{Z}$ at time $t+1$ is updated, independent of the updates happening at all other sites, according to a probability distribution that is supported on $\{0,1\}$ and that depends on $(\eta_{t}(n+i): i \in \mathcal{N})$, i.e.\ on the states $\eta_{t}(n)$, $\eta_{t}(n+1)$ and $\eta_{t}(n+2)$ of the sites $n$, $n+1$ and $n+2$ at time $t$. The precise description of these probability distributions can be found in \S\ref{subsec:PCA}. We remark here that in the sequel, much of our mathematical analysis is carried out on a different but related PCA, $\F_{p,q}$, with alphabet $\{0,?,1\}$, that is referred to as the \emph{envelope} to $F_{p,q}$, since it serves to provide a more direct connection with the percolation games we study in this paper (see \S\ref{subsec:envelope_PCA}). 

We now come to an overview of the various percolation games we are concerned with, in each of which the parameters $p$ and $q$, with $(p,q) \in \mathcal{S}$, are fixed \emph{a priori}, and we begin by assigning, independently to every vertex of $\mathbb{Z}^{2}$, a (random) label that reads \emph{trap} with probability $p$, \emph{target} with probability $q$, and \emph{open} with the remaining probability $r = 1-p-q$. A token is placed at a vertex of $\mathbb{Z}^{2}$, termed the \emph{initial vertex}, at the start of each game, and two players take turns to make moves. A \emph{move} involves relocating the token from its current position $(x,y)$ (say) to any vertex in the set $\Out(x,y)$, where
\begin{enumerate}
\item $\Out(x,y) = \{(x,y+2), (x+1,y+1), (x+2,y)\}$ in the first version (henceforth denoted V1),
\item $\Out(x,y) = \{(x,y+1), (x+1,y+1), (x+2,y+1)\}$ in the second version (henceforth denoted V2), 
\item and $\Out(x,y) = \{(x+1,y), (x,y+1), (x-1,y+2)\}$ in the third version (henceforth denoted V3) of percolation games that we study in this paper. 
\end{enumerate}
In each version, a player loses if she relocates the token to a vertex labeled as a trap, and she wins if she relocates the token to a vertex labeled as a target. The game may also continue forever, with neither player being able to reach a target nor being able to force her opponent to fall into a trap, and in this case, we say that the game results in a draw.

The two primary goals of this paper are to establish Theorem~\ref{thm:main_1} and Theorem~\ref{thm:main_2}, stated below, the first of which concerns itself with the outcome of draw in each of the versions V1, V2 and V3 of percolation games mentioned above, and the second establishes ergodicity (see Definition~\ref{ergodicity_defn}) for the PCA $F_{p,q}$.
\begin{theorem}\label{thm:main_1}
The probability of draw in each version of percolation games described above is $0$ for every $(p,q) \in \mathcal{S}$, where $\mathcal{S}$ is as defined in \eqref{mathcal{S}_defn}.
\end{theorem}
\begin{theorem}\label{thm:main_2}
For all $(p,q) \in \mathcal{S}$, the PCA $F_{p,q}$ is ergodic.
\end{theorem}
We give a brief outline of the essential constituents of the proofs of these theorems. For each $(p,q) \in \mathcal{S}$ and for each of the versions V1, V2 and V3, it is shown, via connections explained in \S\ref{subsec:envelope_PCA}, that the probability of the event that a game starting from the origin in $\mathbb{Z}^{2}$ culminates in a draw is equal to the probability that the symbol $?$ occupies the origin in $\mathbb{Z}$ under a certain stationary distribution (see Definition~\ref{ergodicity_defn}) for $\F_{p,q}$. Theorem~\ref{thm:main_3} is then used to establish that the probability of appearance of $?$ at any given site in $\mathbb{Z}$ is $0$ under every stationary distribution for $\F_{p,q}$, thus yielding the conclusion stated in Theorem~\ref{thm:main_1}. Proposition~\ref{prop:ergodicity_implies_draw_probab_0} guarantees that the probability of draw in any of V1, V2 and V3 is $0$ if and only if the PCA $F_{p,q}$ is ergodic. Theorem~\ref{thm:main_2} is thus established by combining the implications of Theorem~\ref{thm:main_1} and Proposition~\ref{prop:ergodicity_implies_draw_probab_0}.

\subsection{Brief discussion of the literature}\label{subsec:literature} A \emph{cellular automaton}, also referred to as a \emph{deterministic cellular automaton} (and henceforth abbreviated as CA) is a discrete dynamical system that consists of
\begin{enumerate*}
\item a regular network of automata (or \emph{cells}) indexed by $\mathbb{Z}^{d}$ for some $d \in \mathbb{N}$, 
\item a finite set of states $\mathcal{A}$ termed the \emph{alphabet}, 
\item a finite set of indices $\mathcal{N} \subset \mathbb{Z}^{d}$ termed the \emph{neighbourhood},
\item and a \emph{local update rule} $f: \mathcal{A}^{\mathcal{N}} \rightarrow \mathcal{A}$.
\end{enumerate*}
The state (in $\mathcal{A}$) of each cell $\mathbf{x} \in \mathbb{Z}^{d}$ is updated by applying $f$ to the current states of all the cells $\mathbf{x} + \mathbf{y}$ where $\mathbf{y} \in \mathcal{N}$, and this process is repeated at discrete time steps (see \cite{CA_survey} for a detailed survey on CAs). It is not surprising, therefore, that many naturally occurring processes (such as collisions between moving point particles in a regular lattice that serves as a model for fluid dynamics) that are governed by local and homogeneous underlying rules can be efficiently modeled and simulated using CAs. CAs correspond to all those functions on $\mathcal{A}^{\mathbb{Z}^{d}}$ that are continuous with respect to the product topology and that commute with translations (see Corollary 76, Proposition 77 and other relevant discussions in Section 5 of \cite{kari_survey}, and \cite{hedlund}). Moreover, CAs serve as mathematical models for massive parallel computations, their relatively simple update rules allow them to be computationally universal, and they are capable of simulating any Turing machine (\cite{Turing_1, Turing_2, Turing_3, Turing_4, Turing_5}). It is no wonder, consequently, that CAs have emerged as a useful tool for studying computations in natural processes as well as physical aspects and physical limits of computation. 

A \emph{probabilistic cellular automaton} (PCA) is obtained when the corresponding update rule is random. PCAs may be interpreted as discrete-time Markov chains on the state space $\mathcal{A}^{\mathbb{Z}^{d}}$ as well as generalizations of CAs. In a $d$-dimensional PCA, the state $\eta_{t+1}(\mathbf{x})$, at time $t+1$, of a cell $\mathbf{x} \in \mathbb{Z}^{d}$ is a random variable whose probability distribution (supported on $\mathcal{A}$) is governed by the states $\eta_{t}(\mathbf{x}+\mathbf{y})$, at time $t$, of the cells $\mathbf{x}+\mathbf{y}$, for all $\mathbf{y} \in \mathcal{N}$. PCAs are obtained in computer science via \emph{random perturbations} of CAs, and the motivation behind studying them includes 
\begin{enumerate*}
\item investigating the fault-tolerant computational capability of CAs (\cite{fault_1, fault_2}),
\item classifying elementary CAs by using their robustness to errors and perturbations as a discriminating criterion (\cite{robust_1, robust_2}), 
\item the intimate connection PCAs have with Gibbs potentials and Gibbs measures in statistical mechanics (\cite{Gibbs, stat_mech_PCA, stat_mech_PCA_2}) as well as with combinatorial models such as directed animals (\cite{PCA_survey, directed_animals_1, directed_animals_2, directed_animals_3, directed_animals_4}), queues ([\cite{PCA_survey}, \S 4.3]) and percolation (\cite{percolation_1, percolation_2, percolation_3, percolation_4, percolation_5}), 
\item and the crucial role PCAs play in modeling several complex systems appearing in physics, chemistry and biology. 
\end{enumerate*} We refer the reader to \cite{PCA_survey_old}, and the more recent \cite{PCA_survey}, for a detailed survey on how the theory of PCAs has developed over the years. 

In addition, to give the reader a glimpse into the vast and variegated applications of PCAs in various disciplines, we refer to \cite{louis_nardi} that offers an insight into how PCAs are implemented in probability, statistical mechanics, computer sciences, natural sciences, dynamical systems and computational cell biology (for instance, the cellular Potts model, stability of emerging patterns, time to attain stationarity in simulation algorithms, transient regimes etc.). We allude to \cite{stat_mech_PCA} that explores the correspondence between stationary measures for PCAs on $\mathbb{Z}^{d}$ and translation-invariant Gibbs measures for related Hamiltonians on $\mathbb{Z}^{d+1}$, leading to a proof that in high-temperature regimes, stationary measures for PCAs are Gibbsian and thereby establishing uniqueness of these stationary measures and exponential decay of correlations in high-noise regimes and phase transition results in low-noise regimes. We refer to \cite{spin_models_PCA} in which the equivalence between $d$-dimensional PCAs and $(d+1)$-dimensional equilibrium spin models satisfying certain disorder conditions is used to analyze phase diagrams, critical behaviour and universality classes of certain automata, to \cite{spatial_dynamical_systems_PCA} that utilizes PCAs (such as Direct Simulation Monte Carlo and Lattice Gas Cellular Automata) as a universal tool to model the highly complex behaviour of non-linear spatially extended dynamical systems, with emphasis on the effect of fluctuations on the dynamics of such systems, Turing structure formation, effects of hydrodynamic modes on the behaviour of non-linear chemical systems (i.e.\ stirring effects), bifurcation changes in the dynamical regimes of complex systems with restricted geometry or low spatial dimensions, descriptions of chemical systems in microemulsions etc., to \cite{stat_mech_PCA_2} which establishes necessary and sufficient conditions under which PCA rules possess underlying Hamiltonians and are thus reversible, and that even for irreversible rules, continuous ferromagnetic transitions in PCAs with up-down symmetry belong to the universality class of kinetic Ising models. In \cite{mean_field_PCA}, PCAs are compared with mean field models; \cite{astrobio_PCA} argues that PCAs serve as the best quantitative framework for modeling astrobiological history of the Milky Way and its Galactic Habitable Zone; \cite{DNA_PCA} uses PCA to model DNA sequence evolutions (such as that of functional genes during phylogenesis) and foretell the likelihood of specific mutations while taking into account mutation rates that depend on the identity of the neighbouring bases, the combined effects of both left and right neighbours on a cell in paired neighbouring mutations, and the processes of natural selection; \cite{crystal_plasticity_PCA} formulates a framework to simulate dynamic recrystallization in hexagonal close-packed metals and alloys using crystal plasticity based finite element model coupled with a PCA approach. We reiterate to the reader that the references mentioned above form just a minuscule fraction of the diverse ways PCAs have found applications in several branches of mathematics, physics, chemistry and biological sciences.

We now come to a more nuanced discussion, focusing on a few of those articles (keeping in mind that this list too is, by no means, exhaustive) that study PCAs from perspectives closely related to how we approach $F_{p,q}$ and $\F_{p,q}$. We begin our discussion with \cite{holroyd2019percolation}, not only because this serves as the primary inspiration for our work in this paper, but also because percolation games were introduced and studied for the first time in \cite{holroyd2019percolation}. The percolation game considered in \cite{holroyd2019percolation} admits $\Out(x,y) = \{(x+1,y), (x,y+1)\}$ for every $(x,y) \in \mathbb{Z}^{2}$, and whether the probability of a draw is strictly positive or not is shown to be intimately connected with the ergodicity of a family of elementary $1$-dimensional PCAs (see \S\ref{subsec:PCA} for details). On the other hand, certain analogous games played on various directed graphs in higher dimensions (such as on an oriented version of the even sub-lattice of $\mathbb{Z}^{d}$ for all $d \geqslant 3$) are shown to exhibit positive probabilities of draw for suitable values of the parameters involved -- a fact that is established via dimension reduction to a hard-core lattice gas model in dimension $d-1$, and by showing that draws occur whenever the corresponding hard-core model fails to have a unique Gibbs measure. 

In \cite{marcovici_sablik_taati}, it is shown that various families of CAs, such as nilpotent CAs, permutive CAs, gliders, CAs with spreading symbols, surjective CAs and algebraic CAs, are highly unstable against various types of noises, such as zero-range noise, memoryless noise, additive noise, permutive noise and birth-death noise, in the sense that they forget their initial conditions under the slightest postive noise -- a fact that is reflected in the ergodicity of the PCAs that result from these CAs under the application of said noise. \cite{marcovici_sablik_taati} also discusses results on the stronger notion of uniform ergodicity of PCAs as well as spatial mixing, computability and admittance of perfect sampling algorithms for their unique stationary or invariant measures, and the techniques used include couplings, entropy and Fourier analysis. In \cite{casse_markovici}, a family of $1$-dimensional PCAs with memory two is studied as these arise naturally from the $8$-vertex model, directed animals, gaz models, TASEP and various other models of statistical physics -- in such a PCA, the state $\eta_{t+1}(n)$ of the site $n \in \mathbb{Z}$ at time $t+1$ is a random variable whose probability distribution is a function of the states $\eta_{t}(n-1)$ and $\eta_{t}(n+1)$ of its two nearest neighbours $n-1$ and $n+1$ at time $t$, and its own state $\eta_{t-1}(n)$ at time $t-1$. In \cite{casse_markovici}, conditions are proposed under which the invariant measures for these PCAs can be expressed either in a product form or in a Markovian form, ergodicity results that hold in this context are proved, and the phenomenon of reversibility of the stationary space-time diagrams of these PCAs is investigated, leading to the discovery of families of Gibbs random fields on the square lattice that have fascinating geometric and combinatorial properties. In \cite{bresler_guo_polyanskiy}, a computer-assisted proof of ergodicity is provided for two PCAs whose update rules can be respectively expressed as $\eta_{t+1}(n) = \text{BSC}_{p}(\text{NAND}(\eta_{t}(n-1), \eta_{t}(n)))$ (referred to as the vertex noise scenario) and $\eta_{t+1}(n) = \text{NAND}(\text{BSC}_{p}(\eta_{t}(n-1), \eta_{t}(n)))$ (referred to as the edge noise scenario), with $p \in (0,\epsilon)$ for some suitable $\epsilon > 0$. Here, BSC$_{p}$ refers to a binary symmetric channel that takes a bit as input and flips it with probability $p$, leaving it unchanged with probability $1-p$. Similar to the approach adopted in this paper, \cite{bresler_guo_polyanskiy} utilizes the notion of \emph{weight function} or \emph{potential function} (see \S\ref{sec:weight_function} for details) that was introduced in \cite{holroyd2019percolation}, but instead of any explicit manual computations, they implement local feasibility of a suitable polynomial linear programming (PLP) to guarantee the existence of a desired potential function that then helps establish the above-mentioned ergodicity results.

\subsection{Organization of the paper}\label{subsec:org} The percolation games we investigate in this paper are discussed and illustrated in \S\ref{subsec:game}, the PCA $F_{p,q}$ is formally introduced in \S\ref{subsec:PCA}, and its envelope $\F_{p,q}$, along with the deep ties it exhibits with the games, is elaborated on in \S\ref{subsec:envelope_PCA}. The results forming the crux of this paper (apart from Theorem~\ref{thm:main_1}) are stated in Proposition~\ref{prop:ergodicity_implies_draw_probab_0}, Proposition~\ref{prop:ergodicity_equivalence}, Theorem~\ref{thm:main_3} and Theorem~\ref{thm:main_4}, while lemmas required to prove Propositions~\ref{prop:ergodicity_implies_draw_probab_0} and \ref{prop:ergodicity_equivalence} are discussed in \S\ref{sec:prelim_results}. Mathematically, the most significant section of this paper is \S\ref{sec:weight_function}, in which the method of weight functions is employed to prove Theorem~\ref{thm:main_3}, and the pivotal steps via which this is accomplished are outlined in \S\ref{subsec:cylinder_computations} through \S\ref{subsec:compose_5}. 

\section{The principal objects studied in this paper}\label{sec:game&PCA}

\subsection{Our percolation games}\label{subsec:game} We begin by describing, in detail, the versions V1, V2 and V3 of percolation games that we briefly dwelt on in \S\ref{subsec:overview}. The permitted moves in each version are illustrated via the directed (as indicated by the arrowheads), dashed, black lines in Figures~\ref{fig:V1}, \ref{fig:V2} and \ref{fig:V3}.
\begin{figure}[h!]
  \centering
    \includegraphics[width=0.3\textwidth]{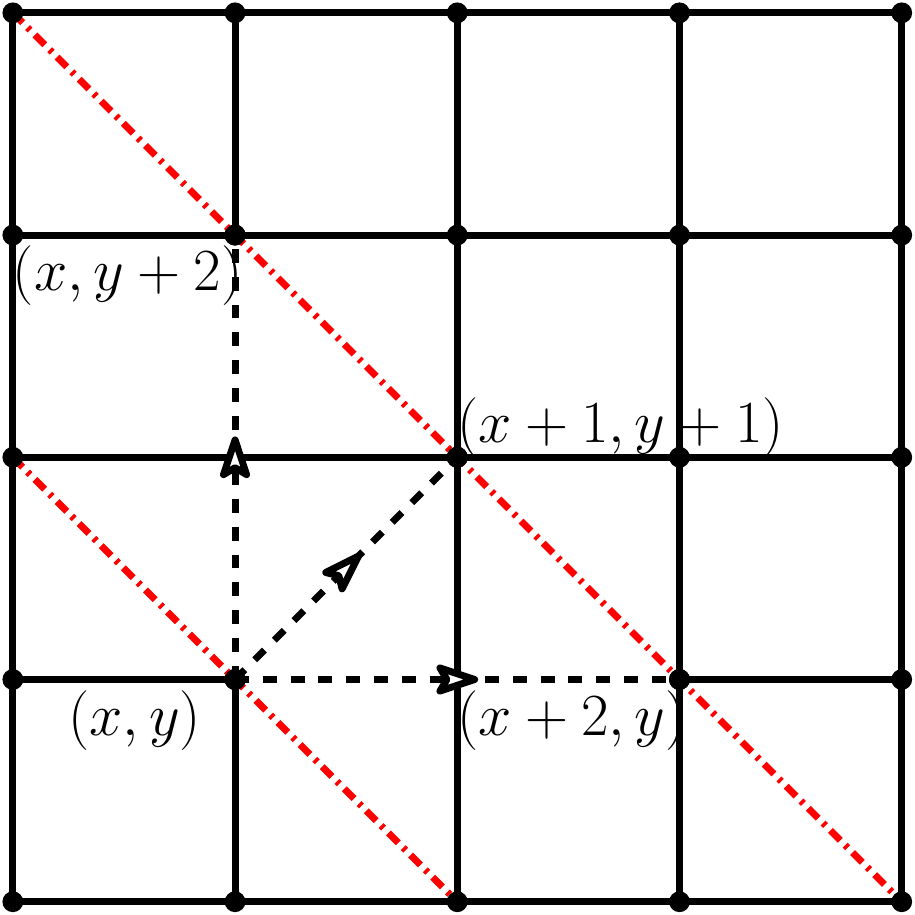}
\caption{$\Out(x,y)$ in version V1}
  \label{fig:V1}
\end{figure}
\begin{figure}[!htb]
    \centering
    \begin{minipage}{.5\textwidth}
        \centering
        \includegraphics[width=0.6\textwidth]{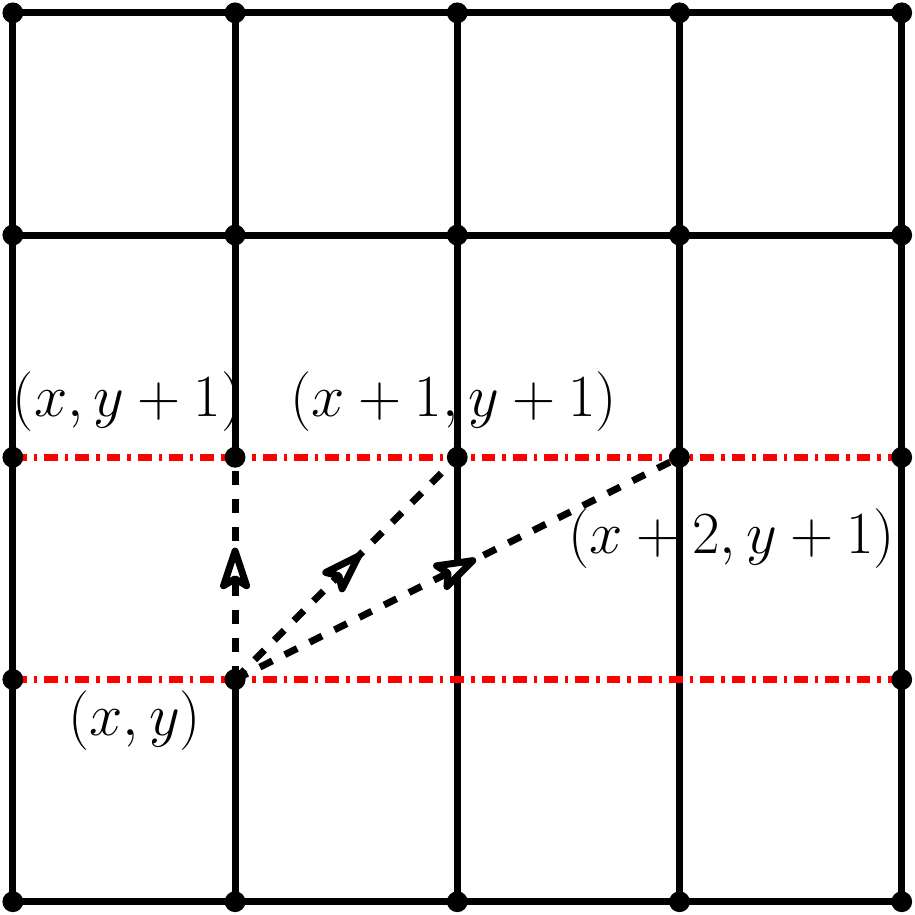}
        \caption{$\Out(x,y)$ in version V2}
        \label{fig:V2}
    \end{minipage}%
    \begin{minipage}{0.5\textwidth}
        \centering
        \includegraphics[width=0.6\textwidth]{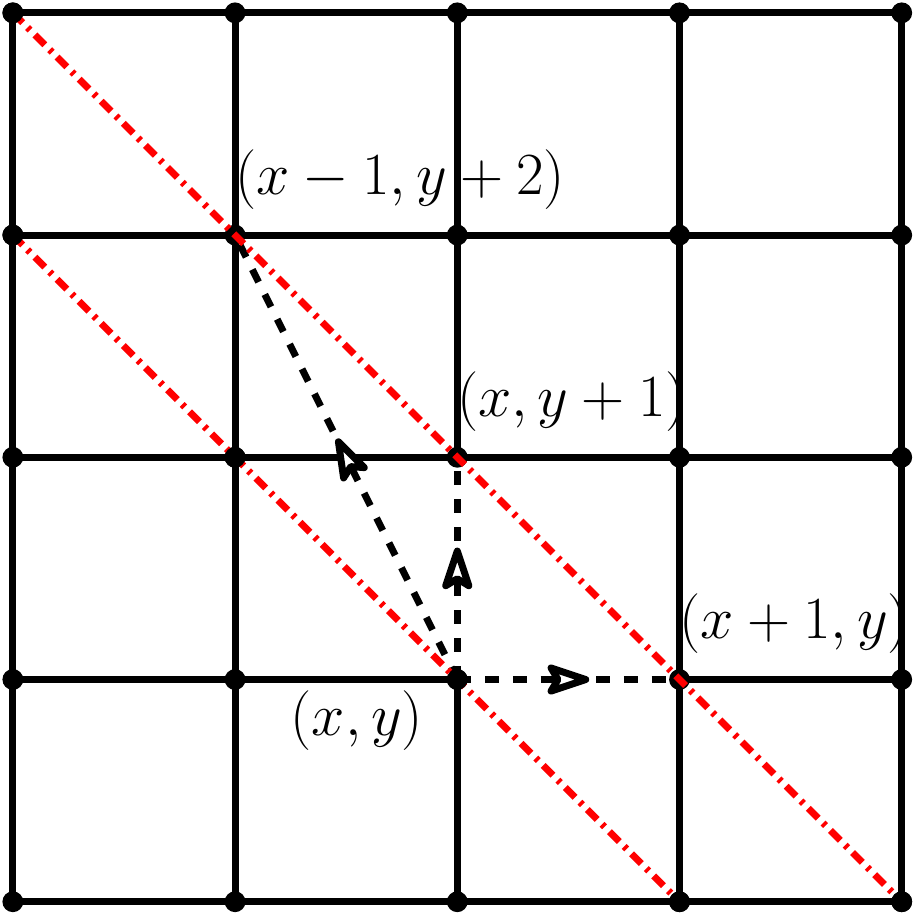}
        \caption{$\Out(x,y)$ in version V3}
        \label{fig:V3}
    \end{minipage}
\end{figure}

We recall from \S\ref{subsec:overview} that each vertex (or site) of $\mathbb{Z}^{2}$ is assigned, independent of all else, a label that is a trap with probability $p$, a target with probability $q$, and open with probability $r = 1-p-q$, where $(p,q) \in \mathcal{S}$. Starting from an initial vertex, the two players take turns to move the token from its current position, say $(x,y)$, to any vertex in $\Out(x,y)$. A player wins if she is able to move the token to a target or if her opponent is forced to move it to a trap. We clarify here that, once $\mathbb{Z}^{2}$ has been endowed with a trap / target / open labeling, to say that the corresponding game is won by a player is to assert that she has a strategy which, when employed, allows her to win no matter what strategy her opponent adopts. The game continues for as long as the token does not land on a site that is marked either a target or a trap, and this could happen indefinitely, leading to a draw. We assume that both players play optimally. The primary question of interest to us is the same as that in \cite{holroyd2019percolation}, i.e.\ for what values of the parameters $p$ and $q$ the game exhibits a positive probability of draw. As previously mentioned, the proof pivots upon the connection each of these games has with the envelope PCA $\F_{p,q}$.

In this context, we mention \cite{basu_holroyd_martin_wastlund} that studies the following two-player combinatorial game on any graph: starting from an initial vertex, the players take turns to move a token, where a move involves relocating the token from the vertex where it is currently situated to a neighbour of that vertex that has not yet been visited. A player who is unable to move loses the game. On $\mathbb{Z}^{2}$, in which odd and even sites, independently, are marked closed (i.e.\ forbidden from being visited by the token) with probabilities $p$ and $q$ respectively, it is shown that the game has probability $0$ of ending in a draw provided closed sites of one parity are sufficiently rare compared to closed sites of the other parity. This question, however, remains open when the percolation parameters $p$ and $q$ are equal. Motivations for studying the games addressed in \cite{basu_holroyd_martin_wastlund} include their deep connections with maximum-cardinality matchings in graphs, and in particular, the ways in which draws in these games relate to sensitivity of such matchings to boundary conditions.

\subsection{Our PCA}\label{subsec:PCA} Recall that we have alluded to the specific PCA $F_{p,q}$ that is of interest to us in \S\ref{subsec:overview}, and included a brief description of PCAs in general in \S\ref{subsec:literature}. A PCA $F$ defined on the lattice $\mathbb{Z}^{d}$ (and hence referred to as a $d$-dimensional PCA), for some $d \in \mathbb{N}$, comprises a finite set of states $\mathcal{A}$ that is called its \emph{alphabet}, a finite set of indices $\mathcal{N} = \{\mathbf{y}_{1}, \mathbf{y}_{2}, \ldots, \mathbf{y}_{m}\} \subset \mathbb{Z}^{d}$ that is called its \emph{neighbourhood}, and a stochastic matrix $\varphi: \mathcal{A}^{m} \times \mathcal{A} \rightarrow [0,1]$ that is called its (random) \emph{local update rule}. Given a \emph{configuration} $\eta$ in the set $\Omega = \mathcal{A}^{\mathbb{Z}^{d}}$, we apply $F$ to $\eta$, obtaining a (random) configuration $F \eta$, in which the state $F \eta(\mathbf{x})$ of the site $\mathbf{x} \in \mathbb{Z}^{d}$ is a random variable whose probability distribution is given by
\begin{align}\label{general_update_rule_eq}
\Prob[F \eta(\mathbf{x}) = b\big|\eta(\mathbf{x}+\mathbf{y}_{i}) = a_{i} \text{ for all } 1 \leqslant i \leqslant m] = \varphi(a_{1}, a_{2}, \ldots, a_{m}, b) \text{ for all } b \in \mathcal{A},
\end{align}  
\sloppy for any $a_{1}, a_{2}, \ldots, a_{m} \in \mathcal{A}$. Here, by definition of stochastic matrices, for all $a_{1}, a_{2}, \ldots, a_{m}, b \in \mathcal{A}$, we have $\varphi(a_{1},a_{2},\ldots,a_{m},b) \geqslant 0$ and $\sum_{b \in \mathcal{A}} \varphi(a_{1},a_{2},\ldots,a_{m},b) = 1$. The updation from $\eta(\mathbf{x})$ to $F \eta(\mathbf{x})$ happens independently over all sites $\mathbf{x} \in \mathbb{Z}^{d}$. Since we consider discrete-time PCAs, it makes sense to indicate by $\eta_{t}$ the configuration at time $t \in \mathbb{N}_{0}$, so that $\eta_{t+1} = F \eta_{t}$ for all $t \in \mathbb{N}_{0}$. We call a PCA \emph{elementary} when it is defined on $\mathbb{Z}$ (i.e.\ $d = 1$) and $|\mathcal{N}| = |\mathcal{A}| = 2$. We refer the reader to [\cite{marcovici_sablik_taati}, \S 2] and \cite{PCA_survey} for excellent expositions on PCAs in general.

Next, we come to a brief discussion regarding the notion of ergodicity of a $d$-dimensional PCA. To begin with, we let $\mathcal{F}$ denote the $\sigma$-field that is generated by the cylinder sets of $\Omega  = \mathcal{A}^{\mathbb{Z}^{d}}$, and we let $\mathbb{D}$ denote the set of all probability measures supported on $\Omega$ and defined with respect to $\mathcal{F}$. We emphasize here that every probability measure on $\Omega$ that is henceforth mentioned belongs to the set $\mathbb{D}$. We define $F^{t}\eta = F(F^{t-1}\eta)$ for $\eta \in \Omega$ and $t \in \mathbb{N}$ (in particular, $F^{1}\eta = F \eta$). In other words, $F^{t}\eta$ is the (random) configuration that is obtained by applying $F$ sequentially $t$ times to the initial configuration $\eta$. These definitions extend naturally to random $\eta$, and if $\eta$ follows the probability distribution $\mu$ (that belongs to $\mathbb{D}$), we let $F^{t} \mu$ (simply written $F\mu$ when $t=1$) denote the probability distribution of the (also random) configuration $F^{t}\eta$. 
\begin{defn}\label{ergodicity_defn}
We say that $\mu$ is a \emph{stationary} or \emph{invariant} measure for a PCA $F$ if $F \mu = \mu$ (in other words, the pushforward measure induced by $F$ is the same as the original measure). We call the PCA $F$ \emph{ergodic} if it has a \emph{unique} stationary measure, say $\mu$, which is \emph{attractive}, i.e.\ for \emph{every} probability measure $\nu$ on $\Omega$, the sequence $F^{t}\nu$ converges weakly to $\mu$ as $t \rightarrow \infty$.
\end{defn}

Consider $d = 1$. Given $y \in \mathbb{Z}$ and a configuration $\eta$, we define the configuration $T^{y} \eta$, with $T^{y}\eta(x) = \eta(x+y)$ for all $x \in \mathbb{Z}$, as the \emph{translation} or \emph{shift} of $\eta$ by $y$. We say that a probability measure $\mu$ belonging to $\mathbb{D}$ is \emph{translation-invariant} or \emph{shift-invariant} if for every subset $B$ measurable with respect to the $\sigma$-field $\mathcal{F}$ introduced earlier, and every $y \in \mathbb{Z}$, we have $\mu(B) = \mu(T^{y}B)$, where $T^{y}B = \{T^{y}\eta: \eta \in B\}$. Given a configuration $\eta$, we denote by $\eta^{R}$, with $\eta^{R}(x) = \eta(-x)$ for all $x \in \mathbb{Z}$, the \emph{reflection} of $\eta$. We say that a probability measure $\mu$ belonging to $\mathbb{D}$ is \emph{reflection-invariant} if for every subset $B$ measurable with respect to $\mathcal{F}$, we have $\mu(B) = \mu(B^{R})$, where $B^{R} = \{\eta^{R}: \eta \in B\}$. These notions will be of use to us in the sequel (see \S\ref{sec:weight_function}). 

We now come to a detailed description of the PCA $F_{p,q}$, with parameters $(p,q) \in \mathcal{S}$. This is a $1$-dimensional PCA, with alphabet $\mathcal{A} = \{0,1\}$ and neighbourhood $\mathcal{N} = \{0,1,2\}$, so that $F_{p,q} \eta(n)$ is a random variable whose probability distribution is a function of $\eta(n)$, $\eta(n+1)$ and $\eta(n+2)$ for all $n \in \mathbb{Z}$. More precisely, slighty tweaking the notation introduced in \eqref{general_update_rule_eq}, the stochastic matrix $\varphi_{p,q}: \mathcal{A}^{3} \times \mathcal{A} \rightarrow [0,1]$ for this PCA is defined via the equations: 
\begin{equation}\label{PCA_rule_1}
 \varphi_{p,q}(0, 0, 0, b) =
  \begin{cases} 
   p & \text{if } b = 0, \\
   1-p & \text{if } b = 1,
  \end{cases}
\end{equation}
and
\begin{equation}\label{PCA_rule_2}
 \varphi_{p,q}(a_{0}, a_{1}, a_{2}, b) =
  \begin{cases} 
   1-q & \text{if } b = 0, \\
   q & \text{if } b = 1,
  \end{cases} \quad \text{for all } (a_{0}, a_{1}, a_{2}) \in \mathcal{A}^{3} \setminus \{(0,0,0)\}.
\end{equation}
The (stochastic) update rule for the automaton $F_{p,q}$ can be illustrated pictorially via Figure~\ref{fig_2}. 

This PCA can be derived from two deterministic CAs, $F_{1}$ and $F_{2}$, whose respective local update rules, $f_{1}$ and $f_{2}$, are given by $f_{1}(a_{0},a_{1},a_{2}) = 1 - \max\{a_{0}, a_{1}, a_{2}\}$ and $f_{2}(a_{0}, a_{1}, a_{2}) = \max\{a_{0}, a_{1}, a_{2}\}$, for all $a_{0}, a_{1}, a_{2} \in \mathcal{A}$. If we define the stochastic matrix $\theta_{1}: \mathcal{A} \times \mathcal{A} \rightarrow [0,1]$ with $\theta_{1}(1,0) = p$, $\theta_{1}(1,1) = 1-p$, $\theta_{1}(0,0) = 1-q$ and $\theta_{1}(0,1) = q$, it is easily verified that $\varphi_{p,q}(a_{0}, a_{1}, a_{2}, b) = \theta_{1}(f_{1}(a_{0}, a_{1}, a_{2}), b)$ for all $a_{0}, a_{1}, a_{2}, b \in \mathcal{A}$. Writing $\theta_{1}(a,b) = (1-\epsilon)\delta_{a}(b) + \epsilon g(b)$ for $a, b \in \mathcal{A}$, where $\epsilon = p+q$, $\delta_{a}(b) = 1$ when $b = a$ and $\delta_{a}(b) = 0$ when $b \neq a$, and $g(0) = 1-g(1) = \frac{p}{p+q}$, we find, referring to \cite{marcovici_sablik_taati}, that our PCA $F_{p,q}$ is obtained from the CA $F_{1}$ via a memoryless zero-range noise. On the other hand, to obtain $F_{p,q}$ from $F_{2}$, we define the stochastic matrix $\theta_{2}: \mathcal{A} \times \mathcal{A} \rightarrow [0,1]$ with $\theta_{2}(0,0) = p$, $\theta_{2}(0,1) = 1-p$, $\theta_{2}(1, 0) = 1-q$ and $\theta_{2}(1,1) = q$, so that $\varphi_{p,q}(a_{0}, a_{1}, a_{2}, b) = \theta_{2}(f_{2}(a_{0}, a_{1}, a_{2}), b)$ for all $a_{0}, a_{1}, a_{2}, b \in \mathcal{A}$.  

\begin{figure}[h!]
  \centering
    \includegraphics[width=0.7\textwidth]{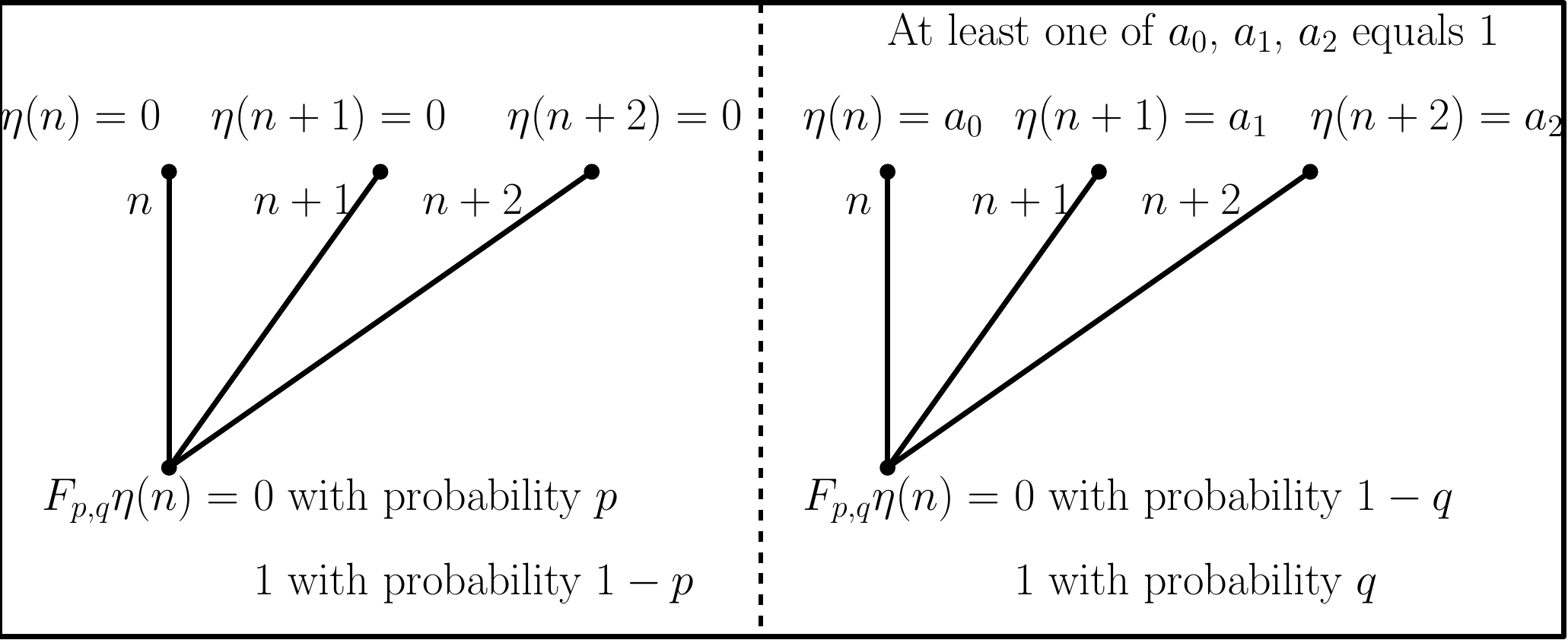}
\caption{The stochastic rules that define our PCA $F_{p,q}$}
  \label{fig_2}
\end{figure}

The PCA $A_{p,q}$ that is studied in \cite{holroyd2019percolation} is an elementary one that bears some resemblance with $F_{p,q}$ in that, following the notation in \eqref{general_update_rule_eq}, its stochastic matrix $\varphi: \mathcal{A}^{2} \times \mathcal{A}$ is defined by the equations:
\begin{equation}
 \varphi(0, 0, b) =
  \begin{cases} 
   p & \text{if } b = 0, \\
   1-p & \text{if } b = 1,
  \end{cases} \quad \text{and} \quad
 \varphi(a_{0}, a_{1}, b) =
  \begin{cases} 
   1-q & \text{if } b = 0, \\
   q & \text{if } b = 1,
  \end{cases} \quad \text{for all } (a_{0}, a_{1}) \in \mathcal{A}^{2} \setminus \{(0,0)\}.\nonumber
\end{equation}
We draw the reader's attention to the primary contrast between $F_{p,q}$ and $A_{p,q}$: whereas in $A_{p,q}$, we draw upon information regarding the states of $2$ consecutive sites $n$ and $n+1$ in order to decide the (random) updated state of the site $n$, in $F_{p,q}$, we draw upon information regarding the states of $3$ consecutive sites $n$, $n+1$ and $n+2$ to decide the (random) updated state of the site $n$. It is, therefore, expected that the latter will have a somewhat-more-involved underlying dependence structure, and establishing ergodicity results for $F_{p,q}$ indeed proves to be a considerably more challenging feat compared to that for $A_{p,q}$, at least as far as using the method of weight functions proposed and implemented in \cite{holroyd2019percolation} is concerned.  

Some discussion on how important the question of ergodicity (or the lack thereof) of PCAs is, and how difficult this question is to resolve under various circumstances, is now in order. It has been found to be notoriously difficult to construct a CA whose trajectories, starting from different initial conditions and evolving under repeated applications of the local update rule, remain distinguishable from each other if even the slightest positive noise is incorporated into the CA (for instance, see \cite{marcovici_sablik_taati}). In other words, most CAs tend to forget their initial conditions under the influence of even the smallest amount of noise, and this is reflected in the ergodicity of the resulting PCA. The renowned \emph{positive rates conjecture} states that \emph{all} PCAs defined on $\mathbb{Z}$ and satisfying $\varphi(a_{1}, a_{2}, \ldots, a_{|\mathcal{N}|}, b) > 0$ for all $a_{1}, a_{2}, \ldots, a_{|\mathcal{N}|}, b \in \mathcal{A}$ (referred to as the \emph{positive rates} condition) are ergodic. An extremely complicated example proposed by \cite{gacs} refutes this conjecture, but the fascinating question still remains as to whether all sufficiently simple, naturally occurring $1$-dimensional PCAs with positive rates are ergodic. There are, however, examples of $d$-dimensional PCAs for $d \geqslant 2$, such as Glauber dynamics for the Ising model at low temperatures, that are known to be non-ergodic. We note here that both $F_{p,q}$ and $A_{p,q}$ have positive rates as long as \emph{both} $p$ and $q$ are strictly positive. 

Multiple sources (see discussions in \cite{casse_markovici}, \cite{holroyd2019percolation} and \cite{bresler_guo_polyanskiy}) reiterate that in general, even if the answer can be guessed from heuristics or simulations, \emph{rigorously} proving whether a given PCA is ergodic or not is a notoriously difficult problem, and is shown to be algorithmically undecidable in \cite{buvsic} and \cite{PCA_survey_old}. Under the assumption of left-right symmetry (which guarantees $\varphi(1,0,0) = \varphi(0,1,0)$), an elementary PCA is determined by the parameters $\varphi(0,0,0)$, $\varphi(1,1,0)$ and $\varphi(1,0,0)$ (recall these notations from \eqref{general_update_rule_eq}). The many existing techniques that have been developed to study ergodicity can take care of ergodicity questions for such PCAs over more than $90\%$ of the volume of the cube $[0,1]^{3}$ defined by these $3$ parameters (\cite{PCA_survey_old}). However, when $p$ and $q$ are small, $A_{p,q}$ belongs to a domain of this cube where none of these techniques works, and this is where \cite{holroyd2019percolation} comes in with their brilliant idea of employing weight functions. To the best of our knowledge, the problem of establishing ergodicity results for $F_{p,q}$, for \emph{all} $(p,q) \in \mathcal{S}$, is an even more challenging one that has, so far, remained open, and we, in this paper, utilize the method of weight functions to provide a concrete proof of Theorem~\ref{thm:main_2}:

\subsection{Our envelope PCA and its relation to our percolation games}\label{subsec:envelope_PCA} We begin by deducing certain recurrence relations that arise naturally in each version of percolation games described in \S\ref{subsec:game}. Once $\mathbb{Z}^{2}$ has been endowed with the trap / target / open labeling, we define a site $(x,y)$ to be in the class $W$ if the game that begins with $(x,y)$ as the initial vertex is won by the player who plays the first round. We define $(x,y)$ to be in the class $L$ if the game that begins with $(x,y)$ as the initial vertex is lost by the player who plays the first round, and we define $(x,y)$ to be in the class $D$ if the game that begins with $(x,y)$ as the initial vertex results in a draw. In particular, if $(x,y)$ is a trap, then we place it in $W$, and if it is a target, we place it in $L$. The intuition behind these conventions is as follows: one may imagine an ``unseen" round that takes place \emph{before} the actual game begins, in which the player who is supposed to play the second round of the actual game moves the token from somewhere else to $(x,y)$. Thus, if $(x,y)$ is a trap (respectively a target), she loses (respectively wins) even before the game begins, implying that the player who plays the first round of the actual game wins (respectively loses).

For every $k \in \mathbb{Z}$, we denote by $D_{k} = \{(x,y) \in \mathbb{Z}^{2}: x+y = k\}$ the (top-left to bottom-right) diagonal line containing all sites whose coordinates sum to $k$, and by $H_{k} = \{(x,k): x \in \mathbb{Z}\}$ the horizontal line containing all sites whose $y$-coordinate equals $k$. These lines have been illustrated in red, according to their relevance in V1, V2 and V3, in Figure~\ref{fig_3}. We observe here that in version V1, if a game starts at an initial vertex $(x_{0},y_{0})$, the token \emph{at all times} stays on diagonals $D_{k}$ where $k$ has the same parity as $(x_{0}+y_{0})$.
\begin{figure}[h!]
  \centering
    \includegraphics[width=0.8\textwidth]{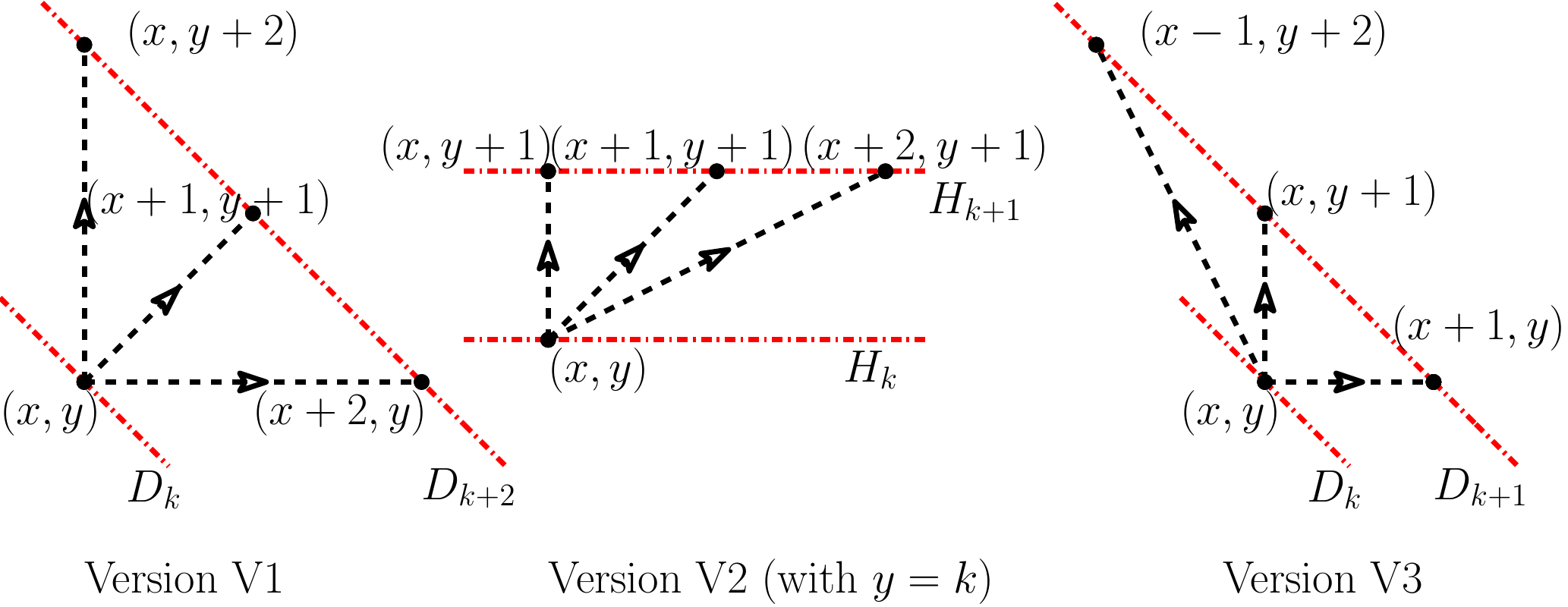}
\caption{Illustrating deduction of the recurrence relations}
  \label{fig_3}
\end{figure}

From the moves permitted in each of V1, V2 and V3, it follows that for any $k \in \mathbb{Z}$, the following are true:
\begin{enumerate}
\item in V1, if all sites $(x,y)$ that lie on the diagonal $D_{k+2}$ have already been categorized into the classes $W$, $L$ and $D$, then this information, along with the pre-assigned labels of trap / target / open, is enough to determine the classes to which the sites lying on $D_{k}$ belong,
\item in V2, if all sites $(x,y)$ that lie on the horizontal line $H_{k+1}$ have already been categorized into the classes $W$, $L$ and $D$, then this information, along with the pre-assigned labels of trap / target / open, is enough to determine the classes to which the sites lying on $H_{k}$ belong,
\item in V3, if all sites $(x,y)$ that lie on the diagonal $D_{k+1}$ have already been categorized into the classes $W$, $L$ and $D$, then this information, along with the pre-assigned labels of trap / target / open, is enough to determine the classes to which the sites lying on $D_{k}$ belong. 
\end{enumerate} 
Using Figure~\ref{fig_3} as a reference, we draw the following conclusions, assuming that $(x,y)$ is our initial vertex:
\begin{enumerate}
\item If each vertex of $\Out(x,y)$ belongs to $W$, then no matter which of these vertices the first player moves the token to from $(x,y)$, the second player wins. We thus have the following two possibilities: 
\begin{enumerate*}
\item either the vertex $(x,y)$ has been marked a trap and hence belongs to $W$, which happens with probability $p$, 
\item or else the game that begins from $(x,y)$ results in a loss for the first player, so that $(x,y)$ is classified into $L$ with the remaining probability $1-p$.
\end{enumerate*}
\item If at least one of the vertices of $\Out(x,y)$ belongs to $L$, the first player moves the token from $(x,y)$ to \emph{this} vertex, making the second player lose. We thus have the following two possibilities:
\begin{enumerate*}
\item either $(x,y)$ has been marked a target and hence belongs to $L$, which happens with probability $q$, 
\item or else the game that begins from $(x,y)$ results in a win for the first player, so that $(x,y)$ is classified into $W$ with the remaining probability $1-q$.
\end{enumerate*}
\item The final scenario is where \emph{none} of the vertices of $\Out(x,y)$ belongs to $L$ but at least one of them belongs to $D$. In this case, we have the following three possibilities:
\begin{enumerate*}
\item either $(x,y)$ has been marked a trap and hence belongs to $W$, which happens with probability $p$,
\item or $(x,y)$ has been marked a target and hence belongs to $L$, which happens with probability $q$, 
\item or else the game that begins from $(x,y)$ results in a draw, thus placing $(x,y)$ in the class $D$ with the remaining probability $r = 1-p-q$.
\end{enumerate*}
\end{enumerate} 
We further note that in version V1, given the classification of the vertices on $D_{k+2}$ into the classes $W$, $L$ and $D$, the (random) class which a vertex lying on $D_{k}$ gets sorted into via the above-mentioned rules is independent of all other vertices on $D_{k}$. Likewise, in version V2 (respectively V3), conditioned on the classification of the vertices on $H_{k+1}$ (respectively $D_{k+1}$) into the classes $W$, $L$ and $D$, the class which a vertex on $H_{k}$ (respectively $D_{k}$) gets sorted into is independent of all other vertices on $H_{k}$ (respectively $D_{k}$). 

The above recurrence relations are the key to establishing a connection between the games in \S\ref{subsec:game} and the envelope PCA $\F_{p,q}$ that we are now ready to define. Let us identify $W$ with the symbol $0$, $L$ with the symbol $1$ and $D$ with the symbol $?$ (in other words, we label a vertex $0$ if it has been classified into $W$, $1$ if it has been classified into $L$, and $?$ if it has been classified into $D$). For any $k \in \mathbb{Z}$, 
\begin{enumerate}
\item in V1, we identify the diagonal $D_{k}$ with the integer line $\mathbb{Z}$ by identifying $(x,k-x)$ on $D_{k}$ with $x$ on $\mathbb{Z}$ (alternatively, $(k-y,y)$ on $D_{k}$ may also be mapped to $y$ on $\mathbb{Z}$, yielding the same PCA),
\item in V2, we identify $H_{k}$ with $\mathbb{Z}$ by identifying $(x,k)$ on $H_{k}$ with $x$ on $\mathbb{Z}$,
\item in V3, we identify $D_{k}$ with $\mathbb{Z}$ by identifying $(k-y,y)$ on $D_{k}$ with $y$ on $\mathbb{Z}$ (see also Remark~\ref{rem:2-directional_PCA}).
\end{enumerate}
This allows us to represent the recurrence relations listed above via a PCA $\F_{p,q}$ that is endowed with the alphabet $\hat{\mathcal{A}} = \{0, 1, ?\}$, the neighbourhood $\mathcal{N} = \{0,1,2\}$, and the stochastic matrix $\widehat{\varphi}_{p,q}: \hat{\mathcal{A}}^{3} \times \hat{\mathcal{A}} \rightarrow [0,1]$ defined via the equations:
\begin{equation}\label{envelope_PCA_rule_1}
 \widehat{\varphi}_{p,q}(0, 0, 0, b) =
  \begin{cases} 
   p & \text{if } b = 0, \\
   1-p & \text{if } b = 1,
  \end{cases}
\end{equation}
\begin{equation}\label{envelope_PCA_rule_2}
 \widehat{\varphi}_{p,q}(a_{0}, a_{1}, a_{2}, b) =
  \begin{cases} 
   1-q & \text{if } b = 0, \\
   q & \text{if } b = 1,
  \end{cases} \quad \text{for all } (a_{0}, a_{1}, a_{2}) \in \hat{\mathcal{A}}^{3} \setminus \{0,?\}^{3}
\end{equation}
and
\begin{equation}\label{envelope_PCA_rule_3}
 \widehat{\varphi}_{p,q}(a_{0}, a_{1}, a_{2}, b) =
  \begin{cases} 
   p & \text{if } b = 0, \\
   q & \text{if } b = 1,\\
   r = 1-p-q & \text{if } b = ?,
  \end{cases} \quad \text{for all } (a_{0}, a_{1}, a_{2}) \in \{0,?\}^{3} \setminus (0,0,0).
\end{equation}

\begin{figure}[h!]
  \centering
    \includegraphics[width=0.7\textwidth]{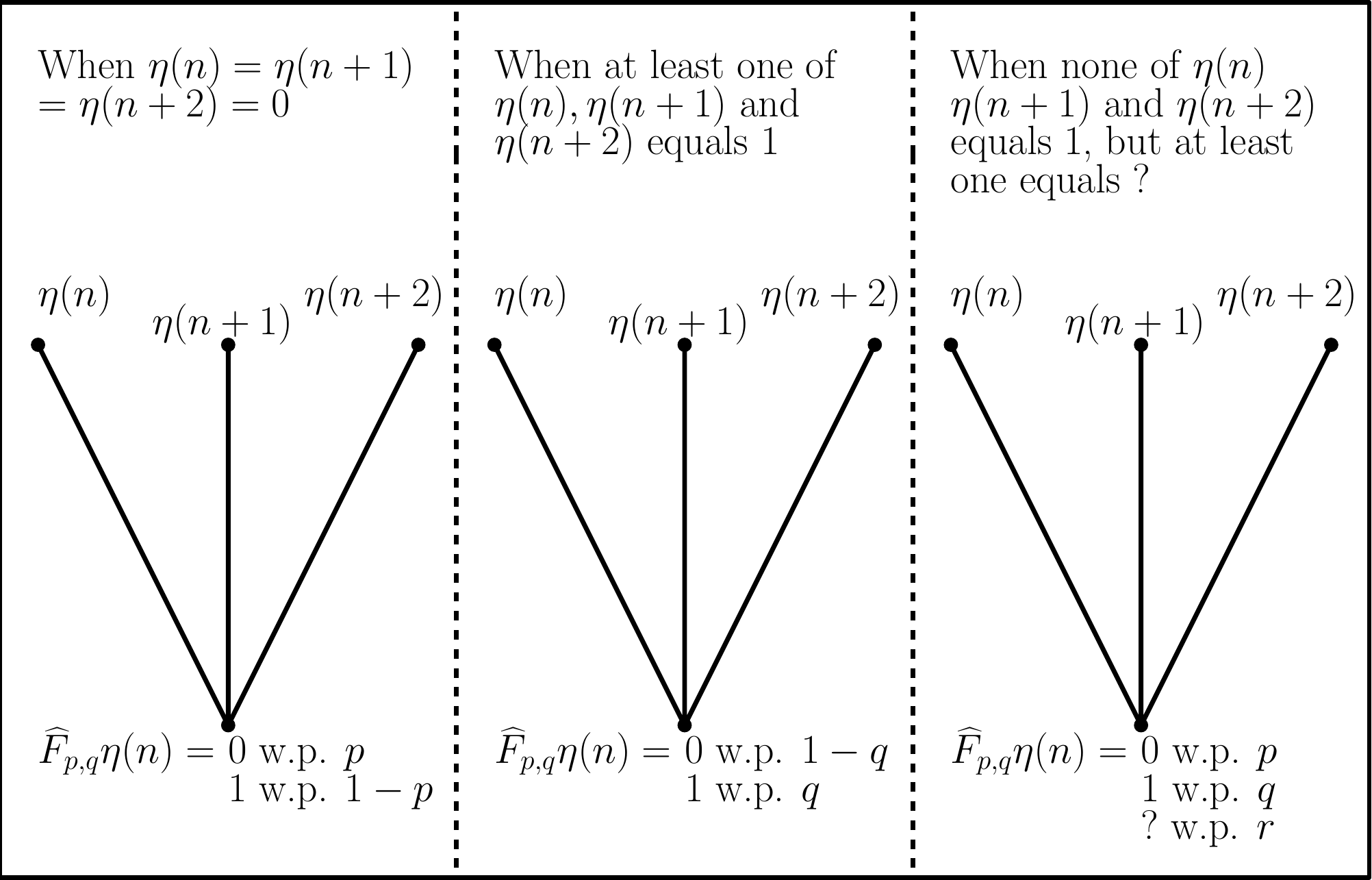}
\caption{The stochastic rules that define our envelope PCA $\F_{p,q}$}
  \label{fig_4}
\end{figure}
We illustrate $\widehat{\varphi}_{p,q}$ in Figure~\ref{fig_4}. To clarify further, in version V1, the classification of the vertices on $D_{k+2}$ into $W$, $L$ and $D$ yields a configuration $\eta \in \hat{\mathcal{A}}^{\mathbb{Z}}$ via the identifications described above, and the classification of the vertices on $D_{k}$ into $W$, $L$ and $D$ via the game's recurrence relations can then be represented by $\F_{p,q}\eta$. Likewise, in version V2 (respectively V3), the classification of the vertices on $H_{k+1}$ (respectively $D_{k+1}$) into $W$, $L$ and $D$ yields a configuration $\eta$, and the classification of the vertices on $H_{k}$ (respectively $D_{k}$) into $W$, $L$ and $D$ is then represented by $\F_{p,q}\eta$.

If we endow $\hat{\mathcal{A}}$ with the total order $0 \prec ? \prec 1$ and define $1-? = ?$, then (similar to the way we interpret $F_{p,q}$ in \S\ref{subsec:PCA}) the envelope PCA $\F_{p,q}$ can be derived, via random perturbations, from two CAs, $\F_{1}$ and $\F_{2}$, with respective local update rules $\f_{1}$ and $\f_{2}$ defined as $\f_{1}(a_{0},a_{1},a_{2}) = 1 - \max\{a_{0},a_{1},a_{2}\}$ and $\f_{2}(a_{0},a_{1},a_{2}) = \max\{a_{0},a_{1},a_{2}\}$. The stochastic matrix $\widehat{\theta}_{1}: \hat{\mathcal{A}} \times \hat{\mathcal{A}} \rightarrow [0,1]$ corresponding to $\F_{1}$ (again, see \cite{marcovici_sablik_taati} for notations) is defined as $\widehat{\theta}_{1}(0,0) = 1-q$, $\widehat{\theta}_{1}(0,1) = \widehat{\theta}_{1}(?,1) = q$, $\widehat{\theta}_{1}(1,0) = \widehat{\theta}_{1}(?,0) = p$, $\widehat{\theta}_{1}(1,1) = 1-p$, $\widehat{\theta}_{1}(?,?) = r$ and $\widehat{\theta}_{1}(0,?) = \widehat{\theta}_{1}(1,?) = 0$, and it is straightforward to check that $\F_{p,q}$ is obtained from $\F_{1}$ by injecting into it a suitable memoryless zero-range noise. The stochastic matrix $\widehat{\theta}_{2}: \hat{\mathcal{A}} \times \hat{\mathcal{A}} \rightarrow [0,1]$ corresponding to $\F_{2}$ is defined as $\widehat{\theta}_{2}(0,0) = \widehat{\theta}_{2}(?,0) = p$, $\widehat{\theta}_{2}(1,1) = \widehat{\theta}_{2}(?,1) = q$, $\widehat{\theta}_{2}(0,1) = 1-p$, $\widehat{\theta}_{2}(1,0) = 1-q$, $\widehat{\theta}_{2}(?,?) = r$ and $\widehat{\theta}_{2}(0,?) = \widehat{\theta}_{2}(1,?) = 0$.

Although we have argued in \S\ref{subsec:envelope_PCA} how the recurrence relations deduced from our percolation games give rise to $\F_{p,q}$, there is, in fact, another angle from which one may motivate the emergence of $\F_{p,q}$ (we mention here that the term ``envelope" was introduced in \cite{buvsic}) -- namely, the use of $\F_{p,q}$ in coupling two (random) configurations obtained via (possibly repeated) applications of the PCA $F_{p,q}$, starting from two different initial configurations. The symbol $?$ is utilized to populate sites whose actual values may differ between these two coupled configurations. 

It is now time to connect the principal components of this paper, namely the games, the PCA $F_{p,q}$, and its envelope $\F_{p,q}$, with one another. This connection is established via the following three results. 
\begin{prop}\label{prop:ergodicity_implies_draw_probab_0}
For every $(p,q) \in \mathcal{S}$, each of V1, V2 and V3, with parameters $p$ and $q$, has probability $0$ of culminating in a draw if and only if the PCA $F_{p,q}$ is ergodic.
\end{prop}
\begin{prop}\label{prop:ergodicity_equivalence}
The PCA $F_{p,q}$ is ergodic if and only if the corresponding envelope PCA $\F_{p,q}$ is ergodic.
\end{prop}
\begin{theorem}\label{thm:main_3}
For each $(p,q) \in \mathcal{S}$, the envelope PCA $\F_{p,q}$ admits no stationary distribution $\mu$ that assigns positive probability to the symbol $?$. To put it formally, the probability of the event $\eta(x) = ?$, where $\eta$ is a random configuration with law $\mu$, is $0$ for every $x \in \mathbb{Z}$.
\end{theorem}

\begin{remark}\label{rem:2-directional_PCA}
Referring to Figure~\ref{fig_3}, we find that in V3, for each $k \in \mathbb{Z}$, the diagonal $D_{k}$ may also be identified with $\mathbb{Z}$ by mapping $(x,k-x)$ on $D_{k}$ to $x$ on $\mathbb{Z}$. This, then, yields a PCA $\widehat{G}_{p,q}$ with alphabet $\hat{\mathcal{A}}$, neighbourhood $\hat{\mathcal{N}} = \{-1,0,1\}$, and the same (stochastic) local update rule $\widehat{\varphi}_{p,q}$ as defined in \eqref{envelope_PCA_rule_1}, \eqref{envelope_PCA_rule_2} and \eqref{envelope_PCA_rule_3}. In other words, given a configuration $\eta$, the probability distribution of $\widehat{G}_{p,q}\eta(n)$, for any $n \in \mathbb{Z}$, is obtained by letting $\widehat{\varphi}_{p,q}$ take $\eta(n-1)$, $\eta(n)$ and $\eta(n+1)$ as its arguments, as opposed to $\eta(n)$, $\eta(n+1)$ and $\eta(n+2)$. Likewise, the PCA $G_{p,q}$ with alphabet $\mathcal{A} = \{0,1\}$, whose envelope is $\widehat{G}_{p,q}$, is obtained by letting $\varphi_{p,q}$ take $\eta(n-1)$, $\eta(n)$ and $\eta(n+1)$ as its arguments (recall \eqref{PCA_rule_1} and \eqref{PCA_rule_2}). It is straightforward to see that $\widehat{G}_{p,q}$ may also be derived from recurrence relations arising from a fourth version of percolation games, V4, in which $\Out(x,y) = \{(x-1,y+1), (x,y+1), (x+1,y+1)\}$ for each $(x,y) \in \mathbb{Z}^{2}$. 
\end{remark}

Remark~\ref{rem:2-directional_PCA} and results analogous to Proposition~\ref{prop:ergodicity_implies_draw_probab_0} together let us conclude that the probability of draw in each of V3 and V4 is $0$ if and only if $G_{p,q}$ is ergodic. This, in conjunction with Theorem~\ref{thm:main_1}, yields
\begin{theorem}\label{thm:main_4}
Each of $G_{p,q}$ and $\widehat{G}_{p,q}$ is ergodic, and the probability of draw in V4 is $0$, for all $(p,q) \in \mathcal{S}$. 
\end{theorem}  

Although we do not elaborate on the proof, it is actually fairly straightforward to verify, implementing the methods and arguments adopted in this paper, that the following is true. For $(p,q) \in \mathcal{S}$ and any $i \in \mathbb{Z}$, let $G^{(i)}_{p,q}$ be the PCA with alphabet $\mathcal{A} = \{0,1\}$, neighbourhood $\mathcal{N}^{(i)} = \{i,i+1,i+2\}$ and stochastic local update rule $\varphi_{p,q}$ as defined in \eqref{PCA_rule_1} and \eqref{PCA_rule_2}. The corresponding envelope PCA $\widehat{G^{(i)}}_{p,q}$ has alphabet $\hat{\mathcal{A}} = \{0,?,1\}$, neighbourhood $\mathcal{N}^{(i)}$ and stochastic local update rule $\widehat{\varphi}_{p,q}$ as defined in \eqref{envelope_PCA_rule_1}, \eqref{envelope_PCA_rule_2} and \eqref{envelope_PCA_rule_3}. Let $U$ and $V$ be percolation games on $\mathbb{Z}^{2}$, similar to the versions described in \S\ref{subsec:game}, with parameters $p$ and $q$ and with $\Out_{U}(x,y) = \{(x+i,y+1), (x+i+1,y+1), (x+i+2,y+1)\}$ and $\Out_{V}(x,y) = \{(x+i,y-i+1), (x+i+1,y-i), (x+i+2,y-i-1)\}$ respectively. Then both $G^{(i)}_{p,q}$ and $\widehat{G^{(i)}}_{p,q}$ are ergodic, and the probability of draw in each of $U$ and $V$ is $0$, for all $(p,q) \in \mathcal{S}$.

\section{A couple of lemmas before we embark on a proof of Theorem~\ref{thm:main_3}}\label{sec:prelim_results}
Recall from \S\ref{subsec:envelope_PCA} that $\hat{\mathcal{A}} = \{0,?,1\}$, and borrowing from the definitions in \S\ref{subsec:PCA}, we let $\mathbb{D}$ denote the set of all probability measures on $\Omega = \hat{\mathcal{A}}^{\mathbb{Z}}$ that are defined with respect to the $\sigma$-field $\mathcal{F}$ generated by the cylinder sets of $\Omega$.
\begin{lemma}\label{lem:1}
Let $\mu$ and $\widetilde{\mu}$ be two probability distributions in $\mathbb{D}$. Let $\preceq$ denote the stochastic domination (on $\mathbb{D}$) with respect to the coordinate-wise total order induced by the ordering $0 \prec ? \prec 1$, and let $\mu \preceq \widetilde{\mu}$. Then $\F_{p,q}\widetilde{\mu} \preceq \F_{p,q}\mu$.
\end{lemma}
\begin{proof}
Instead of starting with two random configurations, let us fix configurations $\eta$ and $\widetilde{\eta}$ in $\hat{\mathcal{A}}^{\mathbb{Z}}$ such that $\eta(x) \preceq \widetilde{\eta}(x)$ for all $x \in \mathbb{Z}$. Using \eqref{envelope_PCA_rule_1}, \eqref{envelope_PCA_rule_2} and \eqref{envelope_PCA_rule_3}, we now compare $\F_{p,q}\eta(x)$ and $\F_{p,q}\widetilde{\eta}(x)$ for each $x \in \mathbb{Z}$:
\begin{enumerate}
\item Suppose $\widetilde{\eta}(n) = \widetilde{\eta}(n+1) = \widetilde{\eta}(n+2) = 0$, which forces $\eta(n) = \eta(n+1) = \eta(n+2) = 0$. Then both $\F_{p,q}\eta(n)$ and $\F_{p,q}\widetilde{\eta}(n)$ have the same distribution.
\item Suppose $(\widetilde{\eta}(n), \widetilde{\eta}(n+1), \widetilde{\eta}(n+2)) \in \{0,?\}^{3} \setminus \{(0,0,0)\}$. In this case, we either have $(\eta(n), \eta(n+1), \eta(n+2)) \in \{0,?\}^{3} \setminus \{(0,0,0)\}$ or we have $(\eta(n), \eta(n+1), \eta(n+2)) = (0,0,0)$  (for instance, if $(\widetilde{\eta}(n), \widetilde{\eta}(n+1), \widetilde{\eta}(n+2)) = (?,?,0)$, we then have $(\eta(n), \eta(n+1), \eta(n+2)) \in \{(0,0,0), (?,0,0), (0,?,0), (?,?,0)\}$). We need only consider the latter case, where
\begin{multline}
\Prob[\F_{p,q}\widetilde{\eta}(n) \succeq ?] = \Prob[\F_{p,q}\widetilde{\eta}(n) = ?] + \Prob[\F_{p,q}\widetilde{\eta}(n) = 1] = r + q \\= 1-p = \Prob[\F_{p,q}\eta(n) = 1] = \Prob[\F_{p,q}\eta(n) \succeq ?], \text{ and}\nonumber
\end{multline}
\begin{align}
\Prob[\F_{p,q}\widetilde{\eta}(n) \succeq 1] = \Prob[\F_{p,q}\widetilde{\eta}(n) = 1] = q \leqslant 1-p = \Prob[\F_{p,q}\eta(n) = 1] = \Prob[\F_{p,q}\eta(n) \succeq 1].\nonumber
\end{align}
\item Finally, suppose $(\widetilde{\eta}(n), \widetilde{\eta}(n+1), \widetilde{\eta}(n+2)) \in \{0,?,1\}^{3} \setminus \{0,?\}^{3}$. The first possibility here is that $(\eta(n), \eta(n+1), \eta(n+2)) \in \{0,?,1\}^{3} \setminus \{0,?\}^{3}$, where both $\F_{p,q}\eta(n)$ and $\F_{p,q}\widetilde{\eta}(n)$ have the same distribution. The second possibility is where $(\eta(n), \eta(n+1), \eta(n+2)) \in \{0,?\}^{3} \setminus \{(0,0,0)\}$ (for instance, $(\widetilde{\eta}(n), \widetilde{\eta}(n+1), \widetilde{\eta}(n+2)) = (1,?,0)$, in which case we could have $(\eta(n), \eta(n+1), \eta(n+2)) \in \{(?,?,0), (0,?,0), (?,0,0)\}$), so that we have
\begin{align}
\Prob[\F_{p,q}\widetilde{\eta}(n) \succeq ?] = \Prob[\F_{p,q}\widetilde{\eta}(n) = 1] = q \leqslant r+q = \Prob[\F_{p,q}\eta(n) = ?] + \Prob[\F_{p,q}\eta(n) = 1] = \Prob[\F_{p,q}\eta(n) \succeq ?]\nonumber
\end{align}
and
\begin{align}
\Prob[\F_{p,q}\widetilde{\eta}(n) \succeq 1] = \Prob[\F_{p,q}\widetilde{\eta}(n) = 1] = q = \Prob[\F_{p,q}\eta(n) = 1] = \Prob[\F_{p,q}\eta(n) \succeq 1].\nonumber
\end{align}
The third and final possibility is where $(\eta(n), \eta(n+1), \eta(n+2)) = (0,0,0)$, in which case we have
\begin{align}
\Prob[\F_{p,q}\widetilde{\eta}(n) \succeq 1] = \Prob[\F_{p,q}\widetilde{\eta}(n) = 1] = q \leqslant 1-p = \Prob[\F_{p,q}\eta(n) = 1] = \Prob[\F_{p,q}\eta(n) \succeq 1].\nonumber
\end{align}  
\end{enumerate}
These observations let us conclude that $\F_{p,q}\eta \succeq \F_{p,q}\widetilde{\eta}$. Given $\mu \preceq \widetilde{\mu}$, we let $\eta$ and $\widetilde{\eta}$ denote coupled random configurations (defined on the same sample space) such that $\eta$ follows $\mu$, $\widetilde{\eta}$ follows $\widetilde{\mu}$ and $\eta \preceq \widetilde{\eta}$ almost surely. The above deduction applied to $\eta$ and $\widetilde{\eta}$ completes the proof.
\end{proof}

\begin{lemma}\label{lem:2}
Let $\mu$ and $\widetilde{\mu}$ be two probability distributions in $\mathbb{D}$. Let $\unlhd$ denote the stochastic domination (on $\mathbb{D}$) with respect to the coordinate-wise partial order induced by $0 \lhd ? \rhd 1$, and let $\mu \unlhd \widetilde{\mu}$. Then $\F_{p,q}\mu \unlhd \F_{p,q}\widetilde{\mu}$.
\end{lemma}
\begin{proof}
As in the proof of Lemma~\ref{lem:1}, we begin with two configurations $\eta$ and $\widetilde{\eta}$ in $\hat{\mathcal{A}}^{\mathbb{Z}}$ with $\eta(x) \unlhd \widetilde{\eta}(x)$ for all $x \in \mathbb{Z}$. The following two situations are to be considered:
\begin{enumerate}
\item Suppose $\eta(n) = \eta(n+1) = \eta(n+2) = 0$ and $(\widetilde{\eta}(n), \widetilde{\eta}(n+1), \widetilde{\eta}(n+2)) \in \{0,?\}^{3} \setminus \{(0,0,0)\}$. Then
\begin{multline}
\Prob[\F_{p,q}\eta(n) \unrhd 1] = \Prob[\F_{p,q}\eta(n) = 1] = 1-p = r+q \\= \Prob[\F_{p,q}\widetilde{\eta}(n) = ?] + \Prob[\F_{p,q}\widetilde{\eta}(n) = 1] = \Prob[\F_{p,q}\widetilde{\eta}(n) \unrhd 1], \text{ and}\nonumber
\end{multline}
\begin{multline}
\Prob[\F_{p,q}\eta(n) \unrhd 0] = \Prob[\F_{p,q}\eta(n) = 0] = p \leqslant r+p \\= \Prob[\F_{p,q}\widetilde{\eta}(n) = ?] + \Prob[\F_{p,q}\widetilde{\eta}(n) = 0] =  \Prob[\F_{p,q}\widetilde{\eta}(n) \unrhd 0].\nonumber
\end{multline}
\item When $(\eta(n),\eta(n+1),\eta(n+2)) \in \{0,?,1\}^{3} \setminus \{0,?\}^{3}$, we either have $(\widetilde{\eta}(n), \widetilde{\eta}(n+1), \widetilde{\eta}(n+2)) \in \{0,?\}^{3} \setminus \{(0,0,0)\}$ or $(\widetilde{\eta}(n), \widetilde{\eta}(n+1), \widetilde{\eta}(n+2)) \in \{0,?,1\}^{3} \setminus \{0,?\}^{3}$ (for instance, if $(\eta(n), \eta(n+1), \eta(n+2)) = (0,1,1)$, we would consider $(\widetilde{\eta}(n), \widetilde{\eta}(n+1), \widetilde{\eta}(n+2)) \in \{(?,1,1), (0,?,1), (0,1,?), (?,?,1), (?,1,?), (0,?,?), (?,?,?)\}$), but we need only consider the former scenario, since in the latter, $\F_{p,q}\eta(n)$ and $\F_{p,q}\widetilde{\eta}(n)$ have the same distribution. Then
\begin{multline}
\Prob[\F_{p,q}\eta(n) \unrhd 1] = \Prob[\F_{p,q}\eta(n) = 1] = q \leqslant r+q \\= \Prob[\F_{p,q}\widetilde{\eta}(n) = ?] + \Prob[\F_{p,q}\widetilde{\eta}(n) = 1] = \Prob[\F_{p,q}\widetilde{\eta}(n) \unrhd 1], \text{ and}\nonumber
\end{multline}
and
\begin{multline}
\Prob[\F_{p,q}\eta(n) \unrhd 0] = \Prob[\F_{p,q}\eta(n) = 0] = 1-q = r+p \\= \Prob[\F_{p,q}\widetilde{\eta}(n) = ?] + \Prob[\F_{p,q}\widetilde{\eta}(n) = 0] =  \Prob[\F_{p,q}\widetilde{\eta}(n) \unrhd 0].\nonumber
\end{multline}
\end{enumerate}
These observations together let us conclude that $\F_{p,q}\eta \unlhd \F_{p,q}\widetilde{\eta}$, which in turn, arguing the same way as in Lemma~\ref{lem:1}, yields the conclusion stated in Lemma~\ref{lem:2}.
\end{proof}

We implement Lemma~\ref{lem:1} and the argument used in proving Proposition 2.1 of \cite{holroyd2019percolation} to establish Proposition~\ref{prop:ergodicity_equivalence}. Next, an argument identical to that used in proving Proposition 2.2 of \cite{holroyd2019percolation} yields a proof of Proposition~\ref{prop:ergodicity_implies_draw_probab_0}. We are now left with the task of proving Theorem~\ref{thm:main_3}.

\section{The method of weight functions and the proof of Theorem~\ref{thm:main_3}}\label{sec:weight_function}
As discussed in \S\ref{subsec:PCA}, it is a non-trivial task to establish, rigorously, the ergodicity of PCAs in general. It has also been explicitly stated in \cite{holroyd2019percolation} that coming up with a suitable weight function or potential function that serves our purpose of proving ergodicity is not an easy feat either. 

Before we proceed further, we recall here the definitions of translation-invariant and reflection-invariant probability measures as discussed right after Definition~\ref{ergodicity_defn}. Given any $1$-dimensional PCA with alphabet $\mathcal{A}$, a finite index set $S = \{y_{1}, y_{2}, \ldots, y_{n}\} \subset \mathbb{Z}$, and symbols $a_{1}$, $a_{2}$, $\ldots$, $a_{n}$ that belong to $\mathcal{A}$, we call $(a_{1} a_{2} \ldots a_{n})_{S} = \{\eta \in \mathcal{A}^{\mathbb{Z}}: \eta(y_{i}) = a_{i} \text{ for all } 1 \leqslant i \leqslant n\}$ a \emph{cylinder set indexed by $S$}. When a probability measure $\mu$ is translation-invariant, and $y_{i} = k+i$ for all $1 \leqslant i \leqslant n$ and for some $k \in \mathbb{Z}$, we denote the measure $\mu((a_{1} a_{2} \ldots a_{n})_{S})$ of the cylinder set $(a_{1} a_{2} \ldots a_{n})_{S}$ by simply $\mu(a_{1} a_{2} \ldots a_{n})$, since the `location' of $S$ (i.e.\ the value of $k$) ceases to be relevant. For instance, if $\eta$ is a random configuration following the law $\mu$, then for \emph{any} $x \in \mathbb{Z}$, we can let $\mu(?)$ indicate the probability of the event $\eta(x) = ?$. Likewise, for \emph{any} $x \in \mathbb{Z}$, we can let $\mu(0?)$ indicate the probability of the event that $\eta(x) = 0$ and $\eta(x+1) = ?$, and so on. When $\mu$ is both translation-invariant and reflection-invariant, $\mu(a_{1} a_{2} \ldots a_{n}) = \mu(a_{n} a_{n-1} \ldots a_{1})$ for all $a_{1}, \ldots, a_{n} \in \mathcal{A}$. For instance, $\mu(10?) = \mu(?01)$, since the former is the probability of the event $\eta(-1) = 1, \eta(0) = 0, \eta(1) = ?$ whereas the latter is the probability of the event $\eta(-1) = ?, \eta(0) = 0, \eta(1) = 1$.

\begin{remark}\label{rem:shift_reflection_invariant_suffices}
Via an argument identical to that outlined at the very beginning of the proof of Proposition 2.3 of \cite{holroyd2019percolation}, and using Lemma~\ref{lem:2}, we conclude that to prove Theorem~\ref{thm:main_3}, it suffices to show that under no translation-invariant and reflection-invariant stationary distribution for $\F_{p,q}$ can the symbol $?$ appear anywhere with positive probability, for each $(p,q) \in \mathcal{S}$ (where $\mathcal{S}$ is as defined in \eqref{mathcal{S}_defn}). Consequently, we confine ourselves to translation-invariant and reflection-invariant probability measures from now on.
\end{remark}
 
We now come to the actual construction of the weight function, which is accomplished via several steps. To begin with, for the sake of brevity, we let $\widehat{**}$ denote the set $\{0,?\}^{2} \setminus \{(0,0)\}$ and $\widehat{***}$ the set $\{0,?\}^{3} \setminus \{(0,0,0)\}$. We may think of $(\widehat{**})$ as representing the cylinder set in which $(\eta(0),\eta(1)) \in \widehat{**}$, and $(\widehat{***})$ as representing the cylinder set in which $(\eta(0),\eta(1),\eta(2)) \in \widehat{***}$. We write $(S_{0} S_{1} \ldots S_{k} \widehat{***} S'_{0} S'_{1} \ldots S'_{k'})$ (likewise, $(S_{1} S_{2} \ldots S_{k} \widehat{**} S'_{1} S'_{2} \ldots S'_{k'})$) to indicate the cylinder set in which $\eta(i) \in S_{i}$ for all $0 \leqslant i \leqslant k$, $(\eta(k+1),\eta(k+2),\eta(k+3)) \in \widehat{***}$, and $\eta(k+4+i) \in S'_{i}$ for all $0 \leqslant i \leqslant k'$, for any subsets $S_{0}, \ldots, S_{k}, S'_{0}, \ldots, S'_{k'}$ of $\{0,?,1\}$. When $S_{i} = \{a_{i}\}$ is a singleton for some $0 \leqslant i \leqslant k$, we replace $S_{i}$ in the above notation by simply $a_{i}$ (and likewise when $S'_{i}$ is a singleton for some $0 \leqslant i \leqslant k'$). 

Keeping the reader's convenience in mind and before we plunge into the intricate technicalities of the weight function derivation, we briefly dwell here on how we plan to accomplish the task at hand, i.e.\ proving Theorem~\ref{thm:main_3}, using the method of weight functions. Following the notations introduced above, we call a cylinder set $(a_{1}, a_{2}, \ldots, a_{n})_{S}$ \emph{$?$-inclusive} if $a_{i} = ?$ for at least one $i \in \{1, \ldots, n\}$. We envisage our weight function $w(\mu)$ to be a linear combination \begin{equation}\label{gen_form_weight}
w(\mu) = \sum_{i=1}^{s}c_{i}\mu(\mathcal{C}_{i})
\end{equation}
of cylinder sets $\mathcal{C}_{1}, \ldots, \mathcal{C}_{s}$, each of which is $?$-inclusive, with $c_{1}, \ldots, c_{s}$ being real constants (that are functions of the parameters $p$ and $q$), and we want it to satisfy an inequality of the form 
\begin{equation}\label{rough_inequality_form}
w(\F_{p,q}\mu) \leqslant w(\mu) - \sum_{i=1}^{s'}c'_{i}\mu(\mathcal{C}'_{i}),
\end{equation}
where $c'_{1}, \ldots, c'_{s'}$ are non-negative real constants (and are, once again, functions of $p$ and $q$) and $\mathcal{C}'_{1}, \ldots, \mathcal{C}'_{s'}$ are, once again, $?$-inclusive cylinder sets. When $\mu$ is stationary for $\F_{p,q}$, we have $w(\F_{p,q}\mu) = w(\mu)$, which then yields $\sum_{i=1}^{s'}c'_{i}\mu(\mathcal{C}'_{i}) = 0$. The constants $c'_{1}, \ldots, c'_{s'}$ are such that, when $p+q > 0$, we can find a \emph{non-empty} $P \subset \{1, \ldots, s'\}$ such that $c'_{i} > 0$ whenever $i \in P$. This, in turn, yields $\mu(\mathcal{C}'_{i}) = 0$ for every $i \in P$ whenever $\mu$ is stationary and $p+q > 0$, and from these, we infer (as detailed in \S\ref{subsec:compose_5}), that $\mu(?) = 0$.

\subsection{Computation of the probabilities of various cylinder sets under the pushforward measure induced by the action of $\F_{p,q}$}\label{subsec:cylinder_computations} Recall from \S\ref{subsec:PCA} the definition of the pushforward measure $\F_{p,q}\mu$. For \emph{any} translation-invariant and reflection-invariant probability measure $\mu$ belonging to $\mathbb{D}$ (recall from \S\ref{sec:prelim_results}), using \eqref{envelope_PCA_rule_3}, we have $\F_{p,q}\mu(?) = r\mu(\widehat{***})$. In what follows, we use \eqref{envelope_PCA_rule_1}, \eqref{envelope_PCA_rule_2} and \eqref{envelope_PCA_rule_3} to compute the relevant probabilities. Since these computations involve rather similar arguments, we explain in detail only the somewhat-less-straightforward ones.

To compute $\F_{p,q}\mu(100?)$, we note that for $(\F_{p,q}\eta(0), \F_{p,q}\eta(1), \F_{p,q}\eta(2), \F_{p,q}\eta(3))$ to equal $(100?)$ for some $\eta \in \hat{\mathcal{A}}^{\mathbb{Z}}$, we require $(\eta(3),\eta(4),\eta(5)) \in \widehat{***}$. If $\eta(0) = \eta(1) = \eta(2) = 0$, the event $\F_{p,q}\eta(0) = 1$ happens with probability $1-p$, whereas if $(\eta(0), \eta(1), \eta(2)) \in \widehat{***}$ or $\eta(0) = 1$ and $(\eta(1), \eta(2)) \in \{0,?\}^{2}$, the event $\F_{p,q}\eta(0) = 1$ happens with probability $q$, and in each of these situations, each of the events $\F_{p,q}\eta(1) = 0$ and $\F_{p,q}\eta(2) = 0$ happens with probability $p$. If $\eta(1) = 1$ and $\eta(2) \in \{0,?\}$, then $\F_{p,q}\eta(0) = 1$ happens with probability $q$, $\F_{p,q}\eta(1) = 0$ happens with probability $1-q$ and $\F_{p,q}\eta(2) = 0$ happens with probability $p$. If $\eta(2) = 1$, then $\F_{p,q}\eta(0) = 1$ happens with probability $q$ and each of $\F_{p,q}\eta(1) = 0$ and $\F_{p,q}\eta(2) = 0$ happens with probability $1-q$. Combining all, we have
\begin{align}\label{F_{p,q}mu(100?)}
\mu(100?) &= (1-p)p^{2}r\mu(000\widehat{***}) + qp^{2}r[\mu(\widehat{***}\widehat{***}) + \mu(1\{0,?\}^{2}\widehat{***})] + q(1-q)pr\mu(1\{0,?\}\widehat{***}) \nonumber\\&+ q(1-q)^{2}r\mu(1\widehat{***}) = (1-p)p^{2}r\mu(000\widehat{***}) + D_{100?} = p^{2}r^{2}\mu(000\widehat{***}) + C_{100?},
\end{align}
where $C_{100?} = qp^{2}r\mu(\widehat{***}) + qpr^{2}\mu(1\{0,?\}\widehat{***}) + qr^{2}(1-q+p)\mu(1\widehat{***})$ and $D_{100?} = qp^{2}r\mu(\widehat{***}\widehat{***}) + qp^{2}r\mu(1\{0,?\}^{2}\widehat{***}) + q(1-q)pr\mu(1\{0,?\}\widehat{***}) + q(1-q)^{2}r\mu(1\widehat{***})$. Arguing likewise, we obtain
\begin{align}
\bullet & \F_{p,q}\mu(0?) = pr\mu(\{0,?\}\widehat{***}) + (1-q)r\mu(1\widehat{***});\label{F_{p,q}mu(0?)}\\
\bullet & \F_{p,q}\mu(?0?) = pr^{2}[\mu(\{0,?\}^{2}\widehat{***}) - \mu(000\widehat{**})];\label{F_{p,q}mu(?0?)}\\
\bullet & \F_{p,q}\mu(1?) = r^{2}\mu(000?) + qr\mu(\widehat{***}).\label{F_{p,q}mu(1?)}
\end{align}

For some cylinder sets, we only need certain parts of the expressions for their probabilities. For instance, while computing $\F_{p,q}\mu(10?)$, we first consider $\eta(0) = \eta(1) = \eta(2) = 0$, so that $(\eta(3), \eta(4)) \in \widehat{**}$, the event $\F_{p,q}\eta(0) = 1$ happens with probability $1-p$, and the event $\F_{p,q}\eta(1) = 0$ happens with probability $p$. Next, we consider $\eta(0) = 1$, $\eta(1) \in \{0,?\}$ and $(\eta(2), \eta(3), \eta(4)) \in \widehat{***}$, so that $\F_{p,q}\eta(0) = 1$ happens with probability $q$ and $\F_{p,q}\eta(1) = 0$ happens with probability $p$, and finally, we consider $\eta(1) = 1$, so that $\F_{p,q}\eta(0) = 1$ happens with probability $q$ and $\F_{p,q}\eta(1) = 0$ happens with probability $1-q$. Therefore
\begin{align}\label{F_{p,q}mu(10?)}
\mu(10?) &= (1-p)pr\mu(000\widehat{**}) + C_{10?} + D_{10?}
\end{align}
where $C_{10?} = qpr\mu(1\{0,?\}\widehat{***}) + q(1-q)r\mu(1\widehat{***})$ and $D_{10?}$ is the component arising from the case where $(\eta(0), \eta(1), \eta(2)) \in \widehat{***}$. Similar arguments lead to
\begin{align}
\bullet & F_{p,q}\mu(1??) = (1-p)r^{2}\mu(000?\{0,?\}) + C_{1??}\label{F_{p,q}mu(1??)};\\
\bullet & F_{p,q}\mu(1?0?) = (1-p)r^{2}p\mu(000?\{0,?\}^{2}) + C_{1?0?}\label{F_{p,q}mu(1?0?)};\\
\bullet & F_{p,q}\mu(10??) = (1-p)pr^{2}\mu(000\widehat{**}\{0,?\}) + C_{10??}\label{F_{p,q}mu(10??)};\\
\bullet & F_{p,q}\mu(1?00) = (1-p)rp^{2}\mu(000?) + (1-p)r^{2}p\mu(000?\{0,?\}1) \nonumber\\&+ (1-p)r^{2}(1+p-q)\mu(000?1) + C_{1?00}\label{F_{p,q}mu(1?00)};\\
\bullet & F_{p,q}\mu(10?0) = (1-p)p^{2}r\mu(000\widehat{**}) + (1-p)pr^{2}\mu(000\widehat{**}1) + C_{10?0}\label{F_{p,q}mu(10?0)};
\end{align}
where $C_{A}$ accounts for the contribution from cases in which $(\eta(0), \eta(1), \eta(2)) \in \hat{\mathcal{A}}^{3} \setminus \{(0,0,0)\}$, with $A$ being any of $1??$, $1?0?$, $10??$, $1?00$ and $10?0$. While computing $\F_{p,q}\mu(1?01)$, we first consider the case where $\eta(3) = \eta(4) = \eta(5) = 0$, so that $(\eta(1), \eta(2)) \in \widehat{**}$, and the event $\F_{p,q}\eta(2) = 0$ happens with probability $p$ while the event $\F_{p,q}\eta(3) = 1$ happens with probability $1-p$. Note that in this situation, no matter what the value of $\eta(0)$ is, the event $\F_{p,q}\eta(0) = 1$ happens with probability $q$. The second possibility we take into account is where $\eta(0) = \eta(1) = \eta(2) = 0$, which then forces $\eta(3) = ?$, and $\F_{p,q}\eta(0) = 1$ happens with probability $1-p$. If $\eta(4) \in \{0,?\}$, the events $\F_{p,q}\eta(2) = 0$ and $\F_{p,q}\eta(3) = 1$ happen with probabilities $p$ and $q$ respectively, whereas if $\eta(4) = 1$, they happen with probabilities $1-q$ and $q$ respectively. Combining all, we have
\begin{align}\label{F_{p,q}mu(1?01)}
\F_{p,q}\mu(1?01) &= qrp(1-p)\mu(\widehat{**}000) + (1-p)rpq\mu(000?\{0,?\}) + (1-p)r(1-q)q\mu(000?1) + C_{1?01}\nonumber\\
&= 2(1-p)rpq\mu(000?) + (1-p)prq\mu(0000?) + (1-p)rq(r-p)\mu(000?1) + C_{1?01},
\end{align}
where $C_{1?01}$ takes into account the contributions from the situations not considered above. We note here, crucially, that each of $C_{1??}, \ldots, C_{1?01}$ and $D_{10?}$ is non-negative.

\subsection{Important inequalities used in the derivation of the weight function}\label{subsec:ineq}
We use Tables~\ref{table_1}, \ref{table_2}, \ref{table_3} and \ref{table_4} to illustrate the derivation of a few inequalities. In each of these tables, the first row indicates the indices of the coordinates in $\mathbb{Z}$, and the rows that follow represent events involving cylinder sets. To elucidate, in Table~\ref{table_1}, the second and third rows represent respectively the events $(\eta(-1),\eta(0),\eta(1),\eta(2)) = (1,?,?,?)$ and $(\eta(-1),\eta(0),\eta(1),\eta(2)) = (1,?,?,0)$, and it is immediate that these two events are disjoint since they disagree on the symbol that occupies the coordinate $2$. In fact, for any two distinct rows that are inside the same table, it can be seen that the corresponding events are mutually exclusive. To give the reader an understanding of how we make use of these tables, note that the union of the pairwise disjoint events listed in Table~\ref{table_1} forms a subset of the event $\eta(0) = ?$. Thus, Table~\ref{table_1}, along with Remark~\ref{rem:shift_reflection_invariant_suffices}, allows us to write 
\begin{align}\label{ineq_1}
\mu(?) &\geqslant \mu(\widehat{***}) - \mu(0??) - \mu(0?0) - \mu(?00) + 2\mu(1\widehat{***}) - 2\mu(100?) - \{\mu(??01) + \mu(0?01)\} \nonumber\\&- \{\mu(?0?1) + \mu(00?1)\} + \{2\mu(1??1) + \mu(1?1) + 2\mu(1?01)\}\nonumber\\ 
&= \mu(\widehat{***}) - \mu(0??) - \mu(0?0) - \mu(0?1) - \mu(?00) - \mu(?01) + 2\mu(1\widehat{***}) - 2\mu(100?) \nonumber\\&+ \{2\mu(1??1) + \mu(1?1) + 4\mu(1?01)\} \nonumber\\  
&= \mu(\widehat{***}) - 2\mu(0?) + \mu(?0?) - 2\mu(100?) + 2\mu(1\widehat{***}) + 2\mu(1??1) + \mu(1?1) + 4\mu(1?01).  
\end{align}
Likewise, Tables~\ref{table_2}, \ref{table_3} and \ref{table_4} together yield
\begin{align}\label{ineq_2}
&\mu(?) + \mu(0?) + \mu(00?) \geqslant \mu(\{0,?\}\widehat{***}) + 2\mu(1\widehat{***}) - \mu(100?) + \mu(1?) + 2\mu(1?01) + \mu(1??1).
\end{align}

\begin{table}[htbp]
\begin{center}
\begin{minipage}[b]{0.32\hsize}\centering
\begin{tabular}{ |c|c|c|c|c| } 
 \hline
-2 & -1 & 0 & 1 & 2 \\ 
 \hline
& 1 & ? & ? & ?  \\
  \hline
& 1 & ? & ? & 0  \\
\hline
& 1 & ? & 0 & ?  \\
\hline
& 1 & ? & 0 & 0  \\
\hline
& ? & ? & ? &  \\
\hline
& ? & ? & 0 &  \\
\hline
0 & 0 & ? &  &  \\
\hline
? & 0 & ? &  &  \\
\hline
? & ? & ? & 1 &  \\ 
\hline
0 & ? & ? & 1 &  \\ 
\hline
1 & 0 & ? & ? &  \\
\hline
1 & 0 & ? & 0 &  \\
\hline
& 1 & ? & ? & 1  \\
\hline
& 1 & ? & 0 & 1  \\
\hline
& 1 & ? & 1 &  \\
\hline
1 & ? & ? & 1 &  \\
\hline
1 & 0 & ? & 1 &  \\
\hline
\end{tabular}
\caption{To establish \eqref{ineq_1}}
\label{table_1}
\end{minipage}
\quad
\begin{minipage}[b]{0.32\hsize}\centering
\begin{tabular}{|c|c|c|c|}
\hline
-1 & 0 & 1 & 2\\
\hline
? & ? & ? & ?\\
\hline
? & ? & ? & 0\\
\hline
? & ? & 0 & ?\\
\hline
? & ? & 0 & 0\\
\hline
0 & ? & ? & ?\\
\hline
0 & ? & ? & 0\\
\hline
0 & ? & 0 & ?\\
\hline
0 & ? & 0 & 0\\
\hline
1 & ? & ? & ?\\
\hline
1 & ? & ? & 0\\
\hline
1 & ? & ? & 1\\
\hline
1 & ? & 0 & ?\\
\hline
1 & ? & 0 & 0\\
\hline
1 & ? & 0 & 1\\
\hline
 & ? & 1 & \\
\hline
? & ? & ? & 1\\
\hline
0 & ? & ? & 1\\
\hline
? & ? & 0 & 1\\
\hline
0 & ? & 0 & 1\\
\hline
\end{tabular}
\caption{To establish \eqref{ineq_2}}
\label{table_2}
\end{minipage}
\begin{minipage}[b]{0.32\hsize}\centering
\begin{tabular}{|c|c|c|c|}
\hline
-1 & 0 & 1 & 2\\
\hline
? & 0 & ? & ?\\
\hline
? & 0 & ? & 0\\
\hline
0 & 0 & ? & ?\\
\hline
0 & 0 & ? & 0\\
\hline
1 & 0 & ? & ?\\
\hline
1 & 0 & ? & 0\\
\hline
? & 0 & ? & 1\\
\hline
0 & 0 & ? & 1\\
\hline
1 & 0 & ? & 1\\
\hline
\end{tabular}
\caption{To establish \eqref{ineq_2}}
\label{table_3}
\begin{tabular}{|c|c|c|c|}
\hline
-1 & 0 & 1 & 2\\
\hline
? & 0 & 0 & ?\\
\hline
0 & 0 & 0 & ?\\
\hline
1 & 0 & 0 & ?\\
\hline
\end{tabular}
\caption{To establish \eqref{ineq_2}}
\label{table_4}
\end{minipage}
\end{center}
\end{table}

\subsection{The first step of composing the weight function}\label{subsec:compose_1} We start by defining 
\begin{equation}\label{initial_weight}
w_{0}(\mu) = \mu(?) + 2\mu(0?) - \mu(?0?) + 2\mu(100?).
\end{equation}
It is not straightforward to explain our intuition behind setting $w_{0}$ in \eqref{initial_weight} as our `initial' choice of weight function (after which it gets tweaked and adjusted in several steps described in the sequel to yield the final weight function). Since we are ultimately interested in $\mu(?)$ when $\mu$ is stationary, and since each $\mathcal{C}_{i}$ in \eqref{gen_form_weight} is $?$-inclusive, it is not too far-fetched to entertain the possibility of starting with $\mathcal{C}_{1} = (?)_{\{0\}}$ (i.e.\ the cylinder set in which $?$ occupies the origin). Thus, the right side of \eqref{rough_inequality_form} contains $\mu(?)$, while the left contains $\F_{p,q}\mu(?) = r\mu(\widehat{***})$. When $p$ and $q$ are both small (intuitively, our task of showing $\mu(?) = 0$ for $\mu$ stationary ought to become harder the smaller $p+q$ gets), $r\mu(\widehat{***})$ is nearly equal to $\mu(\widehat{***})$. From \eqref{ineq_1}, we see that $\mu(?) + 2\mu(0?) - \mu(?0?) + 2\mu(100?)$ serves as an upper bound for $\mu(\widehat{***})$. The appearance of $2\mu(0?)$ in the right side of \eqref{rough_inequality_form} implies, from \eqref{F_{p,q}mu(0?)}, that $2(1-q)r\mu(1\widehat{***})$ appears in the left side of \eqref{rough_inequality_form}, and when $p$ and $q$ are both small, this is nearly the same as $2\mu(1\widehat{***})$. From \eqref{ineq_1}, we see that $\mu(?) + 2\mu(0?) - \mu(?0?) + 2\mu(100?)$, in fact, serves as an upper bound for $\mu(\widehat{***}) + 2\mu(1\widehat{***})$. All these provide ample justification as to why we start with $w_{0}$ in \eqref{initial_weight} as our initial weight function.

From \eqref{initial_weight}, \eqref{F_{p,q}mu(0?)}, \eqref{F_{p,q}mu(?0?)} and \eqref{F_{p,q}mu(100?)}, and applying inequalities \eqref{ineq_1} (to obtain an upper bound for the term $r\mu(\widehat{***}) + 2r\mu(1\widehat{***})$) and \eqref{ineq_2} (to obtain an upper bound for the term $pr\mu(\{0,?\}\widehat{***})$), we have:
\begin{align}\label{w_{0}_ineq}
& w_{0}(\F_{p,q}\mu) = (r + 2qp^{2}r)\mu(\widehat{***}) + 2r\{1 + q^{3} - 2q^{2} - qp^{2})\}\mu(1\widehat{***}) + pr\mu(\{0,?\}\widehat{***}) \nonumber\\&+ (pr +  2qpr^{2})\mu(1\{0,?\}\widehat{***}) + pr(p+q)\mu(\{0,?\}^{2}\widehat{***}) + pr^{2}\mu(000\widehat{**}) + 2p^{2}r^{2}\mu(000\widehat{***}) \nonumber\\
&\leqslant w_{0}(\mu) - [p(1-r)+q][\mu(?) + \mu(0?) + \mu(00?)] - (p+q)\mu(10?) - [p(2-r)+2q]\mu(100?)\nonumber\\& - (2r+pr)\mu(1??1) - r\mu(1?1) - (4r+2pr)\mu(1?01) - pr\mu(1?) + 2qp^{2}r\mu(\widehat{***}) \nonumber\\&- 2r(p + 2q^{2} + qp^{2} - q^{3})\mu(1\widehat{***}) + (pr + 2qpr^{2})\mu(1\{0,?\}\widehat{***}) + pr(p+q)\mu(\{0,?\}^{2}\widehat{***}) \nonumber\\&+ pr^{2}\mu(000\widehat{**}) + 2p^{2}r^{2}\mu(000\widehat{***}).
\end{align}
One may note here how \eqref{w_{0}_ineq} lays down the first of the stepping stones which pave the way towards an inequality that resembles \eqref{rough_inequality_form}. However, there are, at this point, several terms in the right side of \eqref{w_{0}_ineq} (such as $2qp^{2}r\mu(\widehat{***})$, $(pr + 2qpr^{2})\mu(1\{0,?\}\widehat{***})$ etc.), other than $w_{0}(\mu)$, in which the coefficients are non-negative, and this needs to be remedied. This is what we accomplish, via several adjustments and suitable algebraic manipulations in-between, in \S\ref{subsec:compose_2}, \S\ref{subsec:compose_3}, \S\ref{subsec:compose_4} and \S\ref{subsec:compose_5}.

In each of the steps involving the above-mentioned adjustments, the inequality (beginning from \eqref{w_{0}_ineq}) evolves. Suppose we have performed $i$ adjustments so far, and the weight function currently under consideration is denoted by $w_{i}$. Let the inequality that we currently have be given by 
\begin{equation}\label{i-th_inequality}
w_{i}(\F_{p,q}\mu) \leqslant w_{i}(\mu) - \sum_{j=1}^{s_{i}}\alpha_{i,j}\mu(\mathcal{E}_{i,j}),
\end{equation}
where $\alpha_{i,1}, \ldots, \alpha_{i,s_{i}}$ are real constants and $\mathcal{E}_{i,1}, \ldots, \mathcal{E}_{i,s_{i}}$ are cylinder sets. Suppose the $(i+1)$-st adjustment is defined via the equation $w_{i+1}(\mu) = w_{i}(\mu) - \sum_{j=1}^{t_{i}}\beta_{i,j}\mu(\mathcal{G}_{i,j})$, in which $\beta_{i,1}, \ldots, \beta_{i,t_{i}}$ are real constants and $\mathcal{G}_{i,1}, \ldots, \mathcal{G}_{i,t_{i}}$ are cylinder sets. We can now rewrite \eqref{i-th_inequality} as follows:
\begin{align}\label{general_adjustment_form}
& w_{i+1}(\F_{p,q}\mu) + \sum_{j=1}^{t_{i}}\beta_{i,j}\F_{p,q}\mu(\mathcal{G}_{i,j}) \leqslant w_{i+1}(\mu) + \sum_{j=1}^{t_{i}}\beta_{i,j}\mu(\mathcal{G}_{i,j}) - \sum_{j=1}^{s_{i}}\alpha_{i,j}\mu(\mathcal{E}_{i,j})\nonumber\\
\implies & w_{i+1}(\F_{p,q}\mu) \leqslant w_{i+1}(\mu) + \sum_{j=1}^{t_{i}}\beta_{i,j}\mu(\mathcal{G}_{i,j}) - \sum_{j=1}^{t_{i}}\beta_{i,j}\F_{p,q}\mu(\mathcal{G}_{i,j}) - \sum_{j=1}^{s_{i}}\alpha_{i,j}\mu(\mathcal{E}_{i,j}).
\end{align}
This is the commonality shared by all the adjustments described in the sequel, and we refer back, several times, to \eqref{general_adjustment_form} and use it to see how the inequality evolves with each adjustment.

\subsection{The second step of composing the weight function}\label{subsec:compose_2}
The first adjustment to the initial weight function in \eqref{initial_weight} is accomplished by setting
\begin{equation}\label{adjust_1}
w_{1}(\mu) = w_{0}(\mu) - p(p+q)\mu(?) = w_{0}(\mu) - p(1-r)\mu(?).
\end{equation}
Applying \eqref{general_adjustment_form} to this adjustment, we have, on one hand, $p(1-r)\mu(?) - [p(1-r)+q]\mu(?) = -q\mu(?)$, and on the other, 
\begin{align}
&- p(p+q)\F_{p,q}\mu(?) + pr(p+q)\mu(\{0,?\}^{2}\widehat{***}) + (pr + 2qpr^{2})\mu(1\{0,?\}\widehat{***}) \nonumber\\&- 2r(p + 2q^{2} + qp^{2} - q^{3})\mu(1\widehat{***}) \nonumber\\
&= - r(p + 4q^{2} + 4qp^{2} - 2q^{3} + 2p^{2} + 2q^{2}p)\mu(1\widehat{***}) - (pr^{2} + 2qpr^{2})\mu(1\widehat{***}1)+ (pr^{2} + 2qpr^{2})\mu(1000?) \nonumber\\& - (pr^{2} + 2qpr^{2})\mu(1?000),\nonumber
\end{align}
where we make use of the identities $\mu(\widehat{***}) = \mu(1\widehat{***}) + \mu(\{0,?\}^{2}\widehat{***}) + \mu(1\{0,?\}\widehat{***})$ and $\mu(1\{0,?\}\widehat{***}) = \mu(1\widehat{***}\{0,?\}) + \mu(1000?) - \mu(1?000)$. Next, using the identities $\mu(1000?) = \mu(000?) - \mu(?000?) - \mu(0000?)$ and $\mu(000??) + \mu(000?0) = \mu(000?) - \mu(000?1)$ and 
\begin{equation}\label{identity_1}
\mu(000\widehat{***}) = \mu(000?) + \mu(0000?) + \mu(00000?) - \mu(000\widehat{**}1) - \mu(000?1),
\end{equation}
we perform the following algebraic manipulation:
\begin{align}
& (pr^{2} + 2qpr^{2})\mu(1000?) + pr^{2}\mu(000\widehat{**}) - (pr^{2} + 2qpr^{2})\mu(1?000) + 2p^{2}r^{2}\mu(000\widehat{***})\nonumber\\
&= 2pr^{2}(1 + p + 2q)\mu(000?) - 2pr^{2}(1 + 2q + p)\mu(1?000) + 2p^{2}r^{2}[\mu(0000?) + \mu(00000?)]\nonumber\\& - 2p^{2}r^{2}\mu(000\widehat{**}1) - pr^{2}(1 + 2q)\mu(?000?) - 2qpr^{2}\mu(000\widehat{**}).\nonumber
\end{align}
Combining all the findings above, we see that the inequality \eqref{w_{0}_ineq} evolves to 
\begin{align}\label{w_{1}_ineq}
& w_{1}(\F_{p,q}\mu) \leqslant w_{1}(\mu) - q\mu(?) - [p(1-r)+q][\mu(0?) + \mu(00?)] - (p+q)\mu(10?) - [p(2-r)+2q]\nonumber\\&\mu(100?) - (2r+pr)\mu(1??1) - r\mu(1?1) - (4r+2pr)\mu(1?01) - pr\mu(1?) + 2qp^{2}r\mu(\widehat{***}) \nonumber\\& - r(p + 4q^{2} + 4qp^{2} - 2q^{3} + 2p^{2} + 2q^{2}p)\mu(1\widehat{***}) - pr^{2}(1 + 2q)\mu(1\widehat{***}1) - pr^{2}(1 + 2q)\nonumber\\&\mu(?000?) - 2qpr^{2}\mu(000\widehat{**}) + 2pr^{2}(1 + p + 2q)\mu(000?) - 2pr^{2}(1 + 2q + p)\mu(1?000) \nonumber\\&+ 2p^{2}r^{2}[\mu(0000?) + \mu(00000?)] - 2p^{2}r^{2}\mu(000\widehat{**}1).
\end{align}
Again, \eqref{w_{1}_ineq}, much like \eqref{w_{0}_ineq}, forms part of the foundation upon which the derivation of an inequality of the form \eqref{rough_inequality_form} is built. But still, there are terms in the right side of \eqref{w_{1}_ineq} (such as $2pr^{2}(1 + p + 2q)\mu(000?)$, $2p^{2}r^{2}[\mu(0000?) + \mu(00000?)]$ etc.) in which the coefficients are non-negative. Therefore, further adjustments are necessary.

\subsection{The third step of composing the weight function}\label{subsec:compose_3}
The second adjustment is carried out as follows:
\begin{equation}\label{adjust_2}
w_{2}(\mu) = w_{1}(\mu) - \{2pr\{\mu(1?) + \mu(10?)\} + 2p^{2}r\{\mu(1??) + \mu(1?0?) + \mu(10??)\} + 4r\mu(1?01) + 2p\mu(100?)\}.
\end{equation}
Applying \eqref{general_adjustment_form} to this adjustment, and using the identities $\mu(1\widehat{***}) = \mu(1?) + \mu(10?) + \mu(100?) - \mu(1?1) - \mu(1??1) - 2\mu(1?01)$ and $\mu(1???) + \mu(1??0) = \mu(1??) - \mu(1??1)$, we obtain, on one hand,
\begin{align}\label{adjust_2_mu_terms_sum}
&- (p+q)\mu(10?) - [p(2-r)+2q]\mu(100?) - (2r+pr)\mu(1??1) - r\mu(1?1) - (4r+2pr)\mu(1?01) - pr\mu(1?) \nonumber\\&- r(p + 4q^{2} + 4qp^{2} - 2q^{3} + 2p^{2} + 2q^{2}p)\mu(1\widehat{***}) + 2pr\{\mu(1?) + \mu(10?)\} + 2p^{2}r\{\mu(1??) + \mu(1?0?) \nonumber\\&+ \mu(10??)\} + 4r\mu(1?01) + 2p\mu(100?)\nonumber\\
&= - (p+q)\mu(10?) - [p(2-r)+2q]\mu(100?) - (2r+pr)\mu(1??1) - r\mu(1?1) - (4r+2pr)\mu(1?01) - pr\mu(1?) \nonumber\\& - pr\{\mu(1?) + \mu(10?) + \mu(100?) - \mu(1?1) - \mu(1??1) - 2\mu(1?01)\} - 2p^{2}r\{\mu(1??) - \mu(1??1) + \mu(1?0?) \nonumber\\& + \mu(10??)\} - r(4q^{2} + 4qp^{2} - 2q^{3} + 2q^{2}p)\{\mu(1???) + \mu(1??0) + \mu(1?0?) + \mu(10??)\} - r(4q^{2} + 4qp^{2} - 2q^{3} \nonumber\\& + 2p^{2} + 2q^{2}p)\{\mu(1?00) + \mu(10?0) + \mu(100?)\} + 2pr\mu(1?) + 2pr\mu(10?) + 2p^{2}r\{\mu(1??) + \mu(1?0?) \nonumber\\& + \mu(10??)\} + 4r\mu(1?01) + 2p\mu(100?)\nonumber\\
&= -[p(1-r)+q]\mu(10?) - \{2q + r(4q^{2} + 4qp^{2} - 2q^{3} + 2p^{2} + 2q^{2}p)\}\mu(100?) - 2r(1-p^{2})\mu(1??1) \nonumber\\&- r(1-p)\mu(1?1) - r(4q^{2} + 4qp^{2} - 2q^{3} + 2q^{2}p)\{\mu(1???) + \mu(1??0) + \mu(1?0?) + \mu(10??)\} - r(4q^{2} + 4qp^{2} \nonumber\\& - 2q^{3} + 2p^{2} + 2q^{2}p)\{\mu(1?00) + \mu(10?0)\}.
\end{align}
On the other, using \eqref{F_{p,q}mu(100?)}, \eqref{F_{p,q}mu(1?)}, \eqref{F_{p,q}mu(10?)}, \eqref{F_{p,q}mu(1??)}, \eqref{F_{p,q}mu(1?0?)}, \eqref{F_{p,q}mu(10??)}, \eqref{F_{p,q}mu(1?01)}, and applying \eqref{identity_1} as well as the identities 
\begin{enumerate*}
\item $\mu(000\widehat{**}) = \mu(000?) + \mu(0000?) - \mu(000?1)$,
\item $\mu(000?\{0,?\}) = \mu(000?) - \mu(000?1)$,
\item $\mu(000?\{0,?\}^{2}) = \mu(000?) - \mu(000?1) - \mu(000?\{0,?\}1)$
\item and $\mu(000\widehat{**}\{0,?\}) = \mu(000?) + \mu(0000?) - \mu(000?1) - \mu(000\widehat{**}1)$, 
\end{enumerate*}
we have
\begin{align}\label{adjust_2_F_{p,q}_terms_sum}
& -2pr\{\F_{p,q}\mu(1?) + \F_{p,q}\mu(10?)\} - 2p^{2}r\{\F_{p,q}\mu(1??) + \F_{p,q}\mu(1?0?) + \F_{p,q}\mu(10??)\} - 4r\F_{p,q}\mu(1?01) \nonumber\\&- 2p\F_{p,q}\mu(100?)\nonumber\\
&= -[2pr^{3} + 2p^{2}r^{2}(1-p) + 2p^{2}r^{3}(1-p) + 4p^{3}r^{3}(1-p) + 8(1-p)r^{2}pq + 2p^{3}r^{2}]\mu(000?) \nonumber\\&- [2p^{2}r^{2}(1-p) + 2p^{3}r^{3}(1-p) + 4(1-p)pr^{2}q + 2p^{3}r^{2}]\mu(0000?) - 2p^{3}r^{2}\mu(00000?) \nonumber\\&+ [2p^{2}r^{2}(1-p) + 2p^{2}r^{3}(1-p) + 4p^{3}r^{3}(1-p) - 4(1-p)r^{2}q(r-p) + 2p^{3}r^{2}]\mu(000?1) \nonumber\\&+ [4p^{3}r^{3}(1-p) + 2p^{3}r^{2}]\mu(000?\{0,?\}1) + [2p^{3}r^{3}(1-p) + 2p^{3}r^{2}]\mu(0000?1) - \{2pqr^{2}\mu(\widehat{***}) \nonumber\\&+ 2prC_{10?} + 2pC_{100?} + D\},
\end{align}
where $D = 2prD_{10?} + 2p^{2}r(C_{1??} + C_{1?0?} + C_{10??}) + 4rC_{1?01}$. We combine the terms involving $\mu(000?)$ and $\mu(000\widehat{**})$ from \eqref{w_{1}_ineq} with the term involving $\mu(000?)$ from \eqref{adjust_2_F_{p,q}_terms_sum} to get
\begin{align}\label{imp_1}
& 2pr^{2}(1 + p + 2q)\mu(000?) - 2qpr^{2}\mu(000\widehat{**}) - [2pr^{3} + 2p^{2}r^{2}(1-p) + 2p^{2}r^{3}(1-p) + 4p^{3}r^{3}(1-p) \nonumber\\&+ 8(1-p)r^{2}pq + 2p^{3}r^{2}]\mu(000?) \nonumber\\
&= [2pqr^{2}\{5p - 2 + p^{2}\} + 6p^{4}r^{2} - 4p^{4}r^{2}(p+q)]\mu(000?) - 2qpr^{2}\mu(0000?) + 2qpr^{2}\mu(000?1).
\end{align}
Next, we combine the terms involving $\mu(0000?)$ from \eqref{w_{1}_ineq}, \eqref{adjust_2_F_{p,q}_terms_sum} and \eqref{imp_1} to get
\begin{align}\label{imp_2}
& 2p^{2}r^{2}\mu(0000?) - [2p^{2}r^{2}(1-p) + 2p^{3}r^{3}(1-p) + 4(1-p)pr^{2}q + 2p^{3}r^{2}]\mu(0000?) - 2qpr^{2}\mu(0000?)\nonumber\\
&= -(6pqr^{2} - 4p^{2}qr^{2})\mu(0000?) - 2p^{3}r^{3}(1-p)\mu(0000?).
\end{align}
We combine the last term from \eqref{imp_2} with the terms involving $\mu(00000?)$ in \eqref{w_{1}_ineq} and \eqref{adjust_2_F_{p,q}_terms_sum} to get
\begin{align}\label{imp_3}
& - 2p^{3}r^{3}(1-p)\mu(0000?) + 2p^{2}r^{2}\mu(00000?) - 2p^{3}r^{2}\mu(00000?)\nonumber\\
&= 2p^{2}r^{2}(1-p)(1-pr)\mu(00000?) - 2p^{3}r^{3}(1-p)\{\mu(10000?) + \mu(?0000?)\}.
\end{align}
Combining the terms involving $\mu(000?1)$ from \eqref{w_{1}_ineq}, \eqref{adjust_2_F_{p,q}_terms_sum} and \eqref{imp_1}, we get
\begin{align}\label{imp_4}
& [2p^{2}r^{2}(1-p) + 2p^{2}r^{3}(1-p) + 4p^{3}r^{3}(1-p) - 4(1-p)r^{2}q(r-p) + 2p^{3}r^{2}]\mu(000?1) \nonumber\\&- 2pr^{2}(1 + 2q + p)\mu(1?000) + 2qpr^{2}\mu(000?1)\nonumber\\
&= -2pr^{2}\{1 + q - p + 2p^{2} - 2p^{2}r(1-p) + pq(1-p) - p^{3}\}\mu(000?1) - 4(1-p)r^{2}q(r-p)\mu(000?1).
\end{align}
Combining the terms involving $\mu(000?\{0,?\}1)$ from \eqref{w_{1}_ineq} and \eqref{adjust_2_F_{p,q}_terms_sum}, we have
\begin{align}\label{imp_5}  
& [4p^{3}r^{3}(1-p) + 2p^{3}r^{2}]\mu(000?\{0,?\}1) - 2p^{2}r^{2}\mu(000?\{0,?\}1) = -2p^{2}r^{2}(1-p)(1-2pr)\mu(000?\{0,?\}1)
\end{align}
and combining the terms involving $\mu(0000?1)$ from \eqref{w_{1}_ineq} and \eqref{adjust_2_F_{p,q}_terms_sum}, we have
\begin{align}\label{imp_6}
& [2p^{3}r^{3}(1-p) + 2p^{3}r^{2}]\mu(0000?1) - 2p^{2}r^{2}\mu(0000?1) = -2p^{2}r^{2}(1-p)(1-pr)\mu(0000?1).
\end{align}
We now consider the second-last and third-last terms of \eqref{adjust_2_F_{p,q}_terms_sum}, and substituting the mathematical expressions for $C_{10?}$ and $C_{100?}$ from \S\ref{subsec:cylinder_computations}, we have
\begin{align}\label{imp_7}
& -2prC_{10?} - 2pC_{100?} = -2qp^{2}r^{2}\mu(1\{0,?\}\widehat{***}) - 2pq(1-q)r^{2}\mu(1\widehat{***}) - 2qp^{3}r\mu(\widehat{***}) \nonumber\\&- 2qp^{2}r^{2}\mu(1\{0,?\}\widehat{***}) - 2pqr^{2}(1-q+p)\mu(1\widehat{***})\nonumber\\
&= -4qp^{2}r^{2}\mu(1\{0,?\}\widehat{***}) - 2pqr^{2}(2-2q+p)\mu(1\widehat{***}) - 2qp^{3}r\mu(\widehat{***}), 
\end{align} 
so that combining the terms involving $\mu(\widehat{***})$ from \eqref{imp_7}, \eqref{w_{1}_ineq} and \eqref{adjust_2_F_{p,q}_terms_sum}, we have
\begin{align}\label{imp_8}
& 2qp^{2}r\mu(\widehat{***}) - 2pqr^{2}\mu(\widehat{***}) - 2qp^{3}r\mu(\widehat{***}) = 2pqr(p - r - p^{2})\mu(\widehat{***}).
\end{align}
Combining the term involving $\mu(?000?)$ from \eqref{w_{1}_ineq}, the first term of \eqref{imp_2} and the term involving $\mu(1\{0,?\}\widehat{***})$ from \eqref{imp_7}, and writing $\mu(000?) = \mu(?000?) + \mu(0000?) + \mu(1000?)$, we have
\begin{align}\label{imp_9}
& - pr^{2}(1 + 2q)\mu(?000?) - (6pqr^{2} - 4p^{2}qr^{2})\mu(0000?) - 4qp^{2}r^{2}\mu(1\{0,?\}\widehat{***}) \nonumber\\
&= -2p^{2}qr^{2}\mu(000?) - pr^{2}\{1+2q(1-p)\}\mu(?000?) - 6pqr^{2}(1-p)\mu(0000?) - 2p^{2}qr^{2}\mu(1000?) \nonumber\\&- 4qp^{2}r^{2}[\mu(1\{0,?\}\widehat{***})- \mu(1000?)],
\end{align}
so that finally, combining the terms involving $\mu(000?)$ from \eqref{imp_1} and \eqref{imp_9}, we have 
\begin{align}\label{imp_10}
& [2pqr^{2}\{5p - 2 + p^{2}\} + 6p^{4}r^{2} - 4p^{4}r^{2}(p+q)]\mu(000?) - 2p^{2}qr^{2}\mu(000?)\nonumber\\
&= [2pqr^{2}\{4p - 2 + p^{2}\} + 6p^{4}r^{2} - 4p^{4}r^{2}(p+q)]\mu(000?).
\end{align}

From \eqref{w_{1}_ineq}, \eqref{adjust_2_mu_terms_sum} and \eqref{imp_1} through \eqref{imp_10}, we obtain the inequality
\begin{align}\label{w_{2}_ineq}
& w_{2}(\F_{p,q}\mu) \leqslant w_{2}(\mu) - q\mu(?) - [p(1-r)+q][\mu(0?) + \mu(00?) + \mu(10?)] - \{2q + r(4q^{2} + 4qp^{2} - 2q^{3} + 2p^{2} \nonumber\\&+ 2q^{2}p)\}\mu(100?) - 2r(1-p^{2})\mu(1??1) - r(1-p)\mu(1?1) - r(4q^{2} + 4qp^{2} - 2q^{3} + 2q^{2}p)\{\mu(1???) + \mu(1??0) \nonumber\\&+ \mu(1?0?) + \mu(10??)\} - r(4q^{2} + 4qp^{2} - 2q^{3} + 2p^{2} + 2q^{2}p)\{\mu(1?00) + \mu(10?0)\} - pr^{2}(1 + 2q)\mu(1\widehat{***}1) \nonumber\\&+ 2p^{2}r^{2}(1-p)(1-pr)\mu(00000?) - 2p^{3}r^{3}(1-p)\{\mu(10000?) + \mu(?0000?)\} - 2pr^{2}\{1 + q - p + 2p^{2} \nonumber\\&- 2p^{2}r(1-p) + pq(1-p) - p^{3}\}\mu(000?1) - 4(1-p)r^{2}q(r-p)\mu(000?1) - 2p^{2}r^{2}(1-p)(1-2pr)\nonumber\\&\mu(000?\{0,?\}1) - 2p^{2}r^{2}(1-p)(1-pr)\mu(0000?1) - 2pqr^{2}(2-2q+p)\mu(1\widehat{***}) + 2pqr(p - r - p^{2})\mu(\widehat{***}) \nonumber\\&- pr^{2}\{1+2q(1-p)\}\mu(?000?) - 6pqr^{2}(1-p)\mu(0000?) - 2p^{2}qr^{2}\mu(1000?) - 4qp^{2}r^{2}[\mu(1\{0,?\}\widehat{***}) \nonumber\\&- \mu(1000?)] + [2pqr^{2}\{4p - 2 + p^{2}\} + 6p^{4}r^{2} - 4p^{4}r^{2}(p+q)]\mu(000?) - D.
\end{align}

\subsection{The fourth step of composing the weight function}\label{subsec:compose_4}
The third adjustment is carried out as follows:
\begin{equation}\label{adjust_3}
w_{3}(\mu) = w_{2}(\mu) - 2(q+p^{2}r)\mu(100?) - 2p^{2}r\{\mu(1?00) + \mu(10?0)\}.
\end{equation}
Once again, by applying \eqref{general_adjustment_form}, we have, on one hand
\begin{align}\label{adjust_3_mu_terms_update}
&2(q+p^{2}r)\mu(100?) + 2p^{2}r\{\mu(1?00) + \mu(10?0)\} - \{2q + r(4q^{2} + 4qp^{2} - 2q^{3} + 2p^{2} + 2q^{2}p)\}\mu(100?) \nonumber\\&- r(4q^{2} + 4qp^{2} - 2q^{3} + 2p^{2} + 2q^{2}p)\{\mu(1?00) + \mu(10?0)\}\nonumber\\
&= -r(4q^{2} + 4qp^{2} - 2q^{3} + 2q^{2}p)\{\mu(100?) + \mu(1?00) + \mu(10?0)\},
\end{align}
and on the other, using \eqref{F_{p,q}mu(100?)}, \eqref{F_{p,q}mu(1?00)} and \eqref{F_{p,q}mu(10?0)}, and adopting similar algebraic manipulations as those used to obtain \eqref{adjust_2_F_{p,q}_terms_sum}, we get:
\begin{align}\label{adjust_3_F_{p,q}mu_terms_sum}
& - 2(q+p^{2}r)\F_{p,q}\mu(100?) - 2p^{2}r\{\F_{p,q}\mu(1?00) + \F_{p,q}\mu(10?0)\} = - 2(q+p^{2}r)\{(1-p)p^{2}r\mu(000\widehat{***}) \nonumber\\&+ D_{100?}\} - 2p^{2}r\{(1-p)rp^{2}\mu(000?) + (1-p)r^{2}p\mu(000?\{0,?\}1) + (1-p)r^{2}(1+p-q)\mu(000?1) + C_{1?00}\} \nonumber\\&- 2p^{2}r\{(1-p)p^{2}r\mu(000\widehat{**}) + (1-p)pr^{2}\mu(000\widehat{**}1) + C_{10?0}\}\nonumber\\
&= -\{2qp^{2}r^{2} + 2q^{2}p^{2}r + 6p^{4}r^{2}(1-p)\}\mu(000?) - \{2qp^{2}r(1-p) + 4p^{4}r^{2}(1-p)\}\mu(0000?) - \{2qp^{2}r(1-p) \nonumber\\&+ 2p^{4}r^{2}(1-p)\}\mu(00000?) + \{2qp^{2}r(1-p) + 4p^{4}r^{2}(1-p) - 2p^{2}r^{3}(1-p)(1+p-q)\}\mu(000?1) \nonumber\\&+ \{2qp^{2}r(1-p) + 2p^{4}r^{2}(1-p) - 4p^{3}r^{3}(1-p)\}\mu(000?\{0,?\}1) + \{2qp^{2}r(1-p) + 2p^{4}r^{2}(1-p) \nonumber\\&- 2p^{3}r^{3}(1-p)\}\mu(0000?1) - D',
\end{align}
where $D' = 2(q+p^{2}r)D_{100?} + 2p^{2}r\{C_{1?00} + C_{10?0}\}$. Combining the term involving $\mu(000?1)$ from \eqref{w_{2}_ineq}, the term involving $\mu(1?00)$ from \eqref{adjust_3_mu_terms_update}, and the term involving $\mu(000?1)$ from \eqref{adjust_3_F_{p,q}mu_terms_sum}, we have
\begin{align}\label{imp_11}
&- r(4q^{2} + 4qp^{2} - 2q^{3} + 2q^{2}p)\mu(1?00) - 2pr^{2}\{1 + q - p + 2p^{2} - 2p^{2}r(1-p) + pq(1-p) - p^{3}\}\mu(000?1) \nonumber\\&- 4(1-p)r^{2}q(r-p)\mu(000?1) + \{2qp^{2}r(1-p) + 4p^{4}r^{2}(1-p) - 2p^{2}r^{3}(1-p)(1+p-q)\}\mu(000?1)\nonumber\\
&\leqslant - r(4q^{2} + 2qp^{2}(1+p) - 2q^{3} + 2q^{2}p)\mu(1?00) - 2pr^{2}\{1 + q - p + p^{3} + 2p^{2}q(1-p) + pq(1-p)\}\nonumber\\&\mu(000?1) - 4(1-p)r^{2}q(r-p)\mu(000?1) - 2p^{2}r^{3}(1-p)(1+p-q)\mu(000?1). 
\end{align}
If $r-p \geqslant 0$, i.e.\ $2p+q \leqslant 1$, \eqref{imp_11} is as far as we need to go when it comes to the coefficient of $\mu(000?1)$. Otherwise, we continue as follows:
\begin{align}\label{imp_12}
&- r(4q^{2} + 2qp^{2}(1+p) - 2q^{3} + 2q^{2}p)\mu(1?00) - 2pr^{2}\{1 + q - p + p^{3} + 2p^{2}q(1-p) + pq(1-p)\}\mu(000?1) \nonumber\\&+ 4(1-p)r^{2}q(p-r)\mu(000?1) - 2p^{2}r^{3}(1-p)(1+p-q)\mu(000?1)\nonumber\\ 
&\leqslant - r(4q^{2} + 2qp^{2}(1+p) - 2q^{3} + 2q^{2}p)\mu(1?00) - 2pr^{2}\{1 + q - p + p^{3} + 2p^{2}q(1-p) + pq(1-p) \nonumber\\&- 2q(1-p)\}\mu(000?1) - 2p^{2}r^{3}(1-p)(1+p-q)\mu(000?1)\nonumber\\ 
&= - r(4q^{2} + 2qp^{2}(1+p) - 2q^{3} + 2q^{2}p)\mu(1?00) - 2pr^{2}\{r + p^{3} + 2p^{2}q(1-p) + pq(1-p) + 2pq\}\mu(000?1) \nonumber\\&- 2p^{2}r^{3}(1-p)(1+p-q)\mu(000?1).
\end{align}
\begin{remark}\label{rem:imp_11_imp_12_combined}
The crucial thing to note here is that each of the main terms in the upper bounds in both \eqref{imp_11} and \eqref{imp_12} is non-positive. For the sake of brevity, though, we can simply consider $- r(4q^{2} + 2qp^{2}(1+p) - 2q^{3} + 2q^{2}p)\mu(1?00) - 2p^{2}r^{3}(1-p)(1+p-q)\mu(000?1)$ as a common upper bound for both the scenarios $2p+q \leqslant 1$ and $2p+q > 1$.
\end{remark}

Next, we combine the term involving $\mu(1??0)$ from \eqref{w_{2}_ineq}, the term involving $\mu(10?0)$ from \eqref{adjust_3_mu_terms_update}, and the terms involving $\mu(000?\{0,?\}1)$ from \eqref{adjust_3_F_{p,q}mu_terms_sum} and \eqref{w_{2}_ineq}, to get
\begin{align}\label{imp_13}
& -r(4q^{2} + 4qp^{2} - 2q^{3} + 2q^{2}p)\{\mu(1??0) + \mu(10?0)\} - 2p^{2}r^{2}(1-p)(1-2pr)\mu(000?\{0,?\}1) \nonumber\\&+ \{2qp^{2}r(1-p) + 2p^{4}r^{2}(1-p) - 4p^{3}r^{3}(1-p)\}\mu(000?\{0,?\}1)\nonumber\\
&\leqslant -r\{4q^{2} + 2qp^{2}(1+p) - 2q^{3} + 2q^{2}p\}\{\mu(1??0) + \mu(10?0)\} - 2p^{2}r^{2}(1-p)(1-p^{2})\mu(000?\{0,?\}1).
\end{align}
Combining the terms involving $\mu(0000?)$ and $\mu(0000?1)$ from \eqref{adjust_3_F_{p,q}mu_terms_sum} with the term involving $\mu(0000?1)$ from \eqref{w_{2}_ineq}, we have
\begin{align}\label{imp_14}
& - \{2qp^{2}r(1-p) + 4p^{4}r^{2}(1-p)\}\mu(0000?) - 2p^{2}r^{2}(1-p)(1-pr)\mu(0000?1) + \{2qp^{2}r(1-p) \nonumber\\&+ 2p^{4}r^{2}(1-p) - 2p^{3}r^{3}(1-p)\}\mu(0000?1) \leqslant -2p^{2}r^{2}(1-p)(1-p^{2})\mu(0000?1) - 4p^{4}r^{2}(1-p)\mu(0000?).
\end{align}
Combining the terms involving $\mu(000?)$ from \eqref{adjust_3_F_{p,q}mu_terms_sum} and \eqref{w_{2}_ineq} yields
\begin{align}\label{imp_15}
& [2pqr^{2}\{4p - 2 + p^{2}\} + 6p^{4}r^{2} - 4p^{4}r^{2}(p+q)]\mu(000?) - \{2qp^{2}r^{2} + 2q^{2}p^{2}r + 6p^{4}r^{2}(1-p)\}\mu(000?)\nonumber\\
&= [2pqr^{2}\{3p - 2 + p^{2}\} - 2q^{2}p^{2}r + 2p^{5}r^{2} - 4p^{4}qr^{2}]\mu(000?).
\end{align}
We combine the terms involving $\mu(00000?)$ from both \eqref{w_{2}_ineq} and \eqref{adjust_3_F_{p,q}mu_terms_sum}, the term $2p^{5}r^{2}\mu(000?)$ from \eqref{imp_15}, the term involving $\mu(0000?)$ from \eqref{imp_14} and the terms involving $\mu(0?)$ and $\mu(00?)$ from \eqref{w_{2}_ineq}, and use the inequalities $\mu(00000?) \leqslant \mu(0000?)$ and $\mu(000?) \leqslant \mu(00?) \leqslant \mu(0?)$. We consider two situations, the first of which is $1-pr-3p^{2} > 0$, in which case we have
\begin{align}\label{imp_16}
& 2p^{2}r^{2}(1-p)(1-pr)\mu(00000?) - \{2qp^{2}r(1-p) + 2p^{4}r^{2}(1-p)\}\mu(00000?) + 2p^{5}r^{2}\mu(000?) \nonumber\\&- 4p^{4}r^{2}(1-p)\mu(0000?) - [p(1-r)+q][\mu(0?) + \mu(00?)]\nonumber\\
&\leqslant 2p^{2}r^{2}(1-p)(1-pr-3p^{2})\mu(00000?) + 2p^{5}r^{2}\mu(000?) - p^{2}\{\mu(0?) + \mu(00?)\} - 2qp^{2}r(1-p)\mu(00000?) \nonumber\\&- q(1+p)\{\mu(0?) + \mu(00?)\}\nonumber\\
&\leqslant 2p^{2}r^{2}\{1 - p(1-p^{2}) - pr(1-p) - 3p^{2}(1-p)\}\mu(000?) - 2p^{2}\mu(000?) - 2qp^{2}r(1-p)\mu(00000?) \nonumber\\&- q(1+p)\{\mu(0?) + \mu(00?)\}\nonumber\\
&= 2p^{2}[r^{2}\{1 - p(1-p^{2}) - pr(1-p) - 3p^{2}(1-p)\} - 1]\mu(000?) - 2qp^{2}r(1-p)\mu(00000?) \nonumber\\&- q(1+p)\{\mu(0?) + \mu(00?)\}. 
\end{align}
Note here that $r^{2} < 1$ and $1 - p(1-p^{2}) - pr(1-p) - 3p^{2}(1-p) \leqslant 1$, so that the coefficient of $\mu(000?)$ in the right side of \eqref{imp_16} is non-positive. The second situation is $1-pr-3p^{2} \leqslant 0$, in which case we simply observe that $2p^{5}r^{2}\mu(000?) - p^{2}\{\mu(0?) + \mu(00?)\} \leqslant 2p^{2}(p^{3}r^{2}-1)\mu(000?)$, and the coefficient here is non-positive. 
\begin{remark}\label{rem:imp_16_other_case_combined}
The above discussion shows that whether $1-pr-3p^{2} > 0$ or not, the leftmost expression in \eqref{imp_16} is bounded above by $- 2qp^{2}r(1-p)\mu(00000?) - q(1+p)\{\mu(0?) + \mu(00?)\}$.
\end{remark} 
Next, we combine the term $2pqr^{2}\{3p - 2 + p^{2}\}\mu(000?)$ from \eqref{imp_15} with the term $- q(1+p)\{\mu(0?) + \mu(00?)\}$ that remains in the upper bound mentioned in Remark~\ref{rem:imp_16_other_case_combined}, and use $2pr \leqslant (p+r)^{2} \leqslant 1$, to get
\begin{align}\label{imp_17}
& 2pqr^{2}\{3p - 2 + p^{2}\}\mu(000?) - q(1+p)\{\mu(0?) + \mu(00?)\} \nonumber\\
&\leqslant 2pqr^{2}\{3p - 2 + p^{2}\}\mu(000?) - 2pqr\mu(000?) - q(1+p-pr)\{\mu(0?) + \mu(00?)\}\nonumber\\
&= -2pqr^{2}(1-p)(p+2)\mu(000?) - 2pqr(1-2pr)\mu(000?) - q(1+p-pr)\{\mu(0?) + \mu(00?)\}. 
\end{align}

From \eqref{w_{2}_ineq}, and using \eqref{adjust_3_mu_terms_update}, Remark~\ref{rem:imp_11_imp_12_combined}, \eqref{imp_13}, \eqref{imp_14}, \eqref{imp_15}, Remark~\ref{rem:imp_16_other_case_combined} and \eqref{imp_17}, we have:
\begin{align}\label{w_{3}_ineq}
& w_{3}(\F_{p,q}\mu) \leqslant w_{3}(\mu) - q\mu(?) - q(1+p-pr)[\mu(0?) + \mu(00?)] - [p(1-r)+q]\mu(10?) - r(4q^{2} + 4qp^{2} - 2q^{3} \nonumber\\&+ 2q^{2}p)\}\{\mu(100?) + \mu(1???) + \mu(1?0?) + \mu(10??)\} - 2r(1-p^{2})\mu(1??1) - r(1-p)\mu(1?1) \nonumber\\& - r(4q^{2} + 2qp^{2}(1+p) - 2q^{3} + 2q^{2}p)\{\mu(1??0)+ \mu(1?00) + \mu(10?0)\} - pr^{2}(1 + 2q)\mu(1\widehat{***}1) \nonumber\\& - 2p^{3}r^{3}(1-p)\{\mu(10000?) + \mu(?0000?)\} - 2pqr^{2}(2-2q+p)\mu(1\widehat{***}) + 2pqr(p - r - p^{2})\mu(\widehat{***}) \nonumber\\&- pr^{2}\{1+2q(1-p)\}\mu(?000?) - 6pqr^{2}(1-p)\mu(0000?) - 2p^{2}qr^{2}\mu(1000?) - 4qp^{2}r^{2}[\mu(1\{0,?\}\widehat{***}) \nonumber\\&- \mu(1000?)] - 2p^{2}r^{3}(1-p)(1+p-q)\mu(000?1) - 2p^{2}r^{2}(1-p)(1-p^{2})\mu(000\widehat{**}1) \nonumber\\&- \{2pqr^{2}(1-p)(2+p) + 2pqr(1-2pr) + 2p^{2}q^{2}r + 4p^{4}qr^{2}\}\mu(000?) - 2qp^{2}r(1-p)\mu(00000?) - D - D'.
\end{align}

\subsection{The final step of composing the weight function and the desired conclusion}\label{subsec:compose_5}
The final adjustment to the existing weight function is accomplished as follows: $w_{4}(\mu) = w_{3}(\mu) - q\mu(?)$. Via \eqref{general_adjustment_form}, we see that the term involving $\mu(?)$ on the right side of \eqref{w_{3}_ineq} vanishes, whereas the term involving $\mu(\widehat{***})$ now updates to 
\begin{align}
& 2pqr(p - r - p^{2})\mu(\widehat{***}) - qr\mu(\widehat{***}) = [-2pqr^{2} + qr\{2p^{2}(1-p) - 1\}]\mu(\widehat{***}),\nonumber
\end{align}
and all we have to note here is that $2p^{2}(1-p) \leqslant 2p(1-p) \leqslant 1$. This final step transforms \eqref{w_{3}_ineq} into  
\begin{align}\label{w_{4}_ineq}
& w_{4}(\F_{p,q}\mu) \leqslant w_{4}(\mu) - q(1+p-pr)[\mu(0?) + \mu(00?)] - [p(1-r)+q]\mu(10?) - r(4q^{2} + 4qp^{2} - 2q^{3} + 2q^{2}p)\}\nonumber\\&\{\mu(100?) + \mu(1???) + \mu(1?0?) + \mu(10??)\} - 2r(1-p^{2})\mu(1??1) - r(1-p)\mu(1?1)  - r(4q^{2} + 2qp^{2}(1+p) \nonumber\\&- 2q^{3} + 2q^{2}p)\{\mu(1??0)+ \mu(1?00) + \mu(10?0)\} - pr^{2}(1 + 2q)\mu(1\widehat{***}1) - 2p^{3}r^{3}(1-p)\{\mu(10000?) \nonumber\\& + \mu(?0000?)\} - 2pqr^{2}(2-2q+p)\mu(1\widehat{***}) - [2pqr^{2} + qr\{1 - 2p^{2}(1-p)\}]\mu(\widehat{***}) - pr^{2}\{1+2q(1-p)\}\nonumber\\&\mu(?000?) - 6pqr^{2}(1-p)\mu(0000?) - 2p^{2}qr^{2}\mu(1000?) - 4qp^{2}r^{2}[\mu(1\{0,?\}\widehat{***}) - \mu(1000?)] \nonumber\\&- 2p^{2}r^{3}(1-p)(1+p-q)\mu(000?1) - 2p^{2}r^{2}(1-p)(1-p^{2})\mu(000\widehat{**}1) - \{2pqr^{2}(1-p)(2+p) \nonumber\\&+ 2pqr(1-2pr) + 2p^{2}q^{2}r + 4p^{4}qr^{2}\}\mu(000?) - 2qp^{2}r(1-p)\mu(00000?) - D - D'.
\end{align}
We see that \eqref{w_{4}_ineq} has the same form as \eqref{rough_inequality_form} (it is crucial to recall, here, that both $D$ and $D'$ are non-negative) -- in particular, the coefficient of every summand in the right side, other than $w_{4}(\mu)$ itself, is non-positive, as desired. From \eqref{initial_weight}, \eqref{adjust_1}, \eqref{adjust_2}, \eqref{adjust_3} and the final adjustment, we obtain our final weight function:
\begin{multline}\label{final_weight}
w_{3}(\mu) = (1 - p^{2} - pq - q)\mu(?) + 2\mu(0?) - \mu(?0?) + 2r(1-p^{2})\mu(100?) - 2pr\{\mu(1?) + \mu(10?)\} \\- 2p^{2}r\{\mu(1??) + \mu(1?0?) + \mu(10??)\} - 4r\mu(1?01) - 2p^{2}r\{\mu(1?00) + \mu(10?0)\}.
\end{multline}

It is now time to deduce our desired conclusion, i.e.\ $\mu(?)$ whenever $\mu$ is stationary for $\F_{p,q}$, from \eqref{w_{4}_ineq}. This conclusion is immediate when $r = 0$, i.e.\ $p+q = 1$, so that henceforth, we only consider $r > 0$. To begin with, we note that when $\mu$ is stationary for $\F_{p,q}$, we have $w_{4}(\F_{p,q}\mu) = w_{4}(\mu)$, and since every term, other than $w_{4}(\mu)$, in the right side of \eqref{w_{4}_ineq} is non-positive, each of them must equal $0$. In particular, this means that each of $-[2pqr^{2} + qr\{1 - 2p^{2}(1-p)\}]\mu(\widehat{***})$, $-[p(1-r)+q]\mu(10?)$ and $-2p^{2}r^{3}(1-p)(1+p-q)\mu(000?1)$ equals $0$. 

When $q > 0$, the coefficient $-[2pqr^{2} + qr\{1 - 2p^{2}(1-p)\}]$ of $\mu(\widehat{***})$ is strictly negative, which yields $\mu(\widehat{***}) = 0$. Since $\mu(?) = \F_{p,q}\mu(?) = r\mu(\widehat{***})$, we have $\mu(?) = 0$. When $q = 0$ and $p > 0$, the coefficients $-[p(1-r)+q]$ and $-2p^{2}r^{3}(1-p)(1+p-q)$ are both strictly negative, which implies that $\mu(10?) = \mu(000?1) = 0$. The former, via \eqref{F_{p,q}mu(10?)}, yields $\mu(000\widehat{**}) = 0$, and this, together with the latter, yields $\mu(000?) = 0$. We now focus on finding $\F_{p,q}\mu(000?)$. 

In order for $(\F_{p,q}\eta(0), \F_{p,q}\eta(1), \F_{p,q}\eta(2), \F_{p,q}\eta(3))$ to equal $(000?)$, we must have $(\eta(3),\eta(4),\eta(5)) \in \widehat{***}$. If each of $\eta(0), \eta(1), \eta(2)$ belongs to $\{0,?\}$, then each of the events $\F_{p,q}\eta(i) = 0$, for $i = 0, 1, 2$, happens with probability $p$. If $\eta(0) = 1$ and $\eta(1), \eta(2) \in \{0,?\}$, then $\F_{p,q}\eta(0) = 0$ happens with probability $1-q$ and each of $\F_{p,q}\eta(1) = 0$ and $\F_{p,q}\eta(2) = 0$ happens with probability $p$. If $\eta(1) = 1$ and $\eta(2) \in \{0,?\}$, then each of $\F_{p,q}\eta(0) = 0$ and $\F_{p,q}\eta(1) = 0$ happens with probability $1-q$ and $\F_{p,q}\eta(2) = 0$ happens with probability $p$. If $\eta(2) = 1$, then each $\F_{p,q}\eta(i) = 0$, for $i = 0, 1, 2$, happens with probability $1-q$. Combining everything, we have $\F_{p,q}\mu(000?) = p^{3}r\mu(\{0,?\}^{3}\widehat{***}) + (1-q)p^{2}r\mu(1\{0,?\}^{2}\widehat{***}) + (1-q)^{2}pr\mu(1\{0,?\}\widehat{***}) + (1-q)^{3}r\mu(1\widehat{***})$.

When $\mu(000?) = 0$, $p > 0$ and $\mu$ is stationary for $\F_{p,q}$, we conclude from the previous paragraph that $\mu(\{0,?\}^{3}\widehat{***}) = \mu(1\{0,?\}^{2}\widehat{***}) = \mu(1\{0,?\}\widehat{***}) = \mu(1\widehat{***})= 0$. Adding these, we get $\mu(\widehat{***}) = 0$, and as argued above for the case $q > 0$, this implies that $\mu(?) = 0$.

\section{Acknowledgements}
For the duration of this project, Dhruv Bhasin has been supported by Grant No.\ 0203/2/2021/RD-II/3033 awarded by the National Board for Higher Mathematics (NBHM), Sayar Karmakar has been supported by Grant No.\ NSF DMS 2124222 awarded by the National Science Foundation (NSF), and Moumanti Podder has been supported by Grant No.\ SERB CRG/2021/006785 awarded by the Science and Engineering Research Board (SERB). In addition to these funders, the authors express their gratitude towards their respective institutes / university for supporting them in myriad ways and for fostering wonderful environments conducive to collaborations in research.

\bibliography{Percolation_games_bibliography}
\end{document}